\documentclass{amsart}
\usepackage{graphicx}
\usepackage{amssymb, amsfonts, amsmath, amsrefs}
\usepackage[all]{xy}
\usepackage{amsmath}
\usepackage{verbatim}
\usepackage[switch, modulo]{lineno}
\usepackage{hyperref}
\newcommand{\doublewidealign}[3]{\begin{minipage}[t]{#1}\begin{align*}#2\end{
align*}\end{minipage}\begin{minipage}[t]{2.5in}\begin{align*}#3\end{align*}\end{
minipage}}

\newcommand{\Abs}[1]{\left\Vert  #1 \right\Vert}
\newcommand{\abs}[1]{\left| #1 \right|}

\newcommand{\varep}{\varepsilon}

\newcommand{\R}{\mathbb{R}}
\newcommand{\RP}{\mathbb{RP}}

\newcommand{\C}{\mathbb{C}}

\newcommand{\psl}[1][\C]{\rm{PSL}_2(#1)}

\newcommand{\fund}[1]{\pi_1(#1)}

\newcommand{\slpm}[2][]{{\rm SL}^{\pm}_{#1}(#2)}

\newcommand{\pgl}{\rm{PGL}_{4}(\R)}
\newcommand{\PGL}[2][]{\rm{PGL}_{#1}(#2)}
\newcommand{\GL}[2][]{\rm{GL}_{#1}(#2)}
\newcommand{\so}[1]{{\rm SO}(#1,1)}

\newcommand{\rvar}[2][\pgl]{\mathfrak{R}(#2,#1)}
\newcommand{\charvar}[2][\pgl]{\mathfrak{X}(#2,#1)}

\newcommand{\Z}{\mathbb{Z}}
\newcommand{\gold}[1]{\mathfrak{B}(#1)}
\newcommand{\HH}{\mathbb{H}}

\newcommand{\dhilb}[3][\Omega]{d_{#1}(#2,#3)}
\newcommand{\hilbvol}[2][\Omega]{\mu_{#1}(#2)}

\newcommand{\bs}{\backslash}

\newcommand{\be}{\begin{enumerate}}
\newcommand{\ee}{\end{enumerate}}
\newcommand{\bd}{\begin{description}}
\newcommand{\ed}{\end{description}}

\newcommand{\ol}[1]{\overline{#1}}

\newcommand{\rhogeo}{\rho_{{\rm geo}}}
\newtheorem{theorem}{Theorem}
\newtheorem{proposition}[theorem]{Proposition}
\newtheorem{corollary}[theorem]{Corollary}
\newtheorem{lemma}[theorem]{Lemma}
\newtheorem{remark}[theorem]{Remark}

\numberwithin{equation}{section}
\numberwithin{theorem}{section}

\begin{document}

\title{Finite Volume Properly Convex Deformations of the Figure-eight Knot}
\author{Samuel A. Ballas}
\email{sballas@math.ucsb.edu}
\date{\today}
\address{Department of Mathematics\\ 
University of California Santa Barbara\\ CA 93106}

\maketitle

\begin{abstract}
 In this paper we show that some open set of the representations of the fundamental group of figure-eight knot complement found in \cite{Ballas12a} are the holonomies of a family of finite volume properly convex projective structures on the figure-eight knot complement. 
\end{abstract}

\section{Introduction}

Let $M$ be an orientable hyperbolizable $n$-manifold with fundamental group $\Gamma$. When $n$ is at least 3, Mostow-Prasad rigidity tells us that if $M$ admits a complete finite volume hyperbolic structures then this structure is unique, up to isometry. Such a structure gives rise to a discrete faithful holonomy representation $\rhogeo:\Gamma\to \text{Isom}^+(\HH^n)$ (unique up to conjugation in $\text{Isom}(\HH^n)$) and $M$ can be realized as $\HH^n/\rhogeo(\Gamma)$. The representation coming from this structure is called the \emph{geometric representation}.

We can view this construction in the projective setting: Let $\Omega\subset \RP^n$ be the projectivization of the negative cone of a quadratic form of signature $(n,1)$. Then $\HH^n$ can be identified with $\Omega$ in such a way that elements of $\text{Isom}^+(\HH^n)$ correspond to  $\text{PSO}(n,1)$. Using this identification, $M$ can be realized as $\Omega/\rhogeo(\Gamma)$.  \emph{Properly convex projective structures} offer us a flexible way to generalize this construction. In this setting we no longer insist that $\Omega$ is the projectivization of the negative cone of a form. Instead we ask only that $\Omega$ is properly convex and that $M$ can be realized as $\Omega/\rho(\Gamma)$, where $\rho:\Gamma\to \PGL[n+1]{\R}$ is a discrete and faithful representation whose image preserves $\Omega$.  

There is no notion of Mostow rigidity in this setting and so the deformation theory of convex projective structures is richer than its hyperbolic counterpart in the following sense. Let $\gold{M}$ be the set of (equivalence classes of) properly convex projective structures on $M$. Mostow rigidity tells us that there is a distinguished point $N_{hyp}\in\gold{M}$ corresponding to the unique complete hyperbolic structure on $M$.  There are examples of manifolds for which $\gold{M}$ has positive dimension at $N_{hyp}$ (see \cite{Goldman90,Marquis12b,CooperLongThist06,PortiHeusener09,Ballas12a} for examples). In contrast to these results, there are many hyperbolic manifolds whose hyperbolic structure cannot be deformed to a non-hyperbolic convex projective structure \cite{Ballas12a,CooperLongThist07}. Said another way, it is possible to find $M$ such that $N_{hyp}$ is an isolated point of $\gold{M}$. 

An important principle in the proof of the results in the previous paragraph is that conjugacy classes of representations of $\Gamma$ into $\PGL[n+1]{\R}$ near $\rhogeo$ give a local parameterization of (equivalence classes of) projective structures on $M$. When $M$ is closed, work of Koszul \cite{Koszul68} shows that the projective structures near $N_{hyp}$ corresponding to conjugacy classes near $\rhogeo$ are properly convex. 

When $M$ is non-compact Koszul's result is not longer true. In general there are representations near $\rhogeo$ that are not the holonomy of any point in $\gold{M}$. An example of such representations are the holonomies of incomplete hyperbolic structures on non-compact $3$-manifolds. Fortunately, recent work of Cooper and Long \cite{CooperLong13} has shown that if the nearby representation satisfies some mild hypotheses then it is the holonomy of a (possibly not properly convex) projective structure on $M$ with nice behavior near the boundary. Furthermore, Cooper, Long, and Tillmann \cite{CooperLongTillman13} have shown that if the restriction of the representation to the peripheral subgroups (fundamental groups of the boundary components) satisfies slightly stronger hypotheses then the representation is the holonomy of a properly convex projective structure.

In \cite{Ballas12a} the author showed that when $M$ is the figure-eight knot complement there exists an explicit family of representations, $\rho_s:\Gamma\to\pgl$ that satisfy the following conditions.
\begin{itemize}
 \item $\rho_0=\rhogeo$,
 \item the $\rho_s$ are pairwise non-conjugate, and 
 \item $\rho_s$ maps the meridian of $M$ to a unipotent element of $\pgl$ and the longitude to a non-unipotent element of $\pgl$.
\end{itemize}

By applying results from \cite{CooperLongTillman13} and carefully analyzing the structure of the cusps we are able to prove the main theorem of this paper.

\begin{theorem}\label{t:maintheorem}
 Let $M$ be the complement in $S^3$ of the figure-eight knot. There exists $\varep$ such that for each $s\in(-\varep,\varep)$, $\rho_s$ is the holonomy of a finite volume properly convex projective structure on $M$. Furthermore, when $s\neq 0$ this structure is not strictly convex.  
\end{theorem}

\begin{remark}
 In recent work with D.\ Long \cite{BallasLong14}, we are able to show that Theorem \ref{t:maintheorem} holds for all $s\in \R$. 
\end{remark}

Rephrasing Theorem \ref{t:maintheorem} in terms of $\gold{M}$ we have the following immediate corollary.

\begin{corollary}\label{c:structurecurve} 
Let $M$ be the complement in $S^3$ of the figure-eight knot. Then there is a non-trivial curve $c:(-\varep,\varep)\to \gold{M}$ such that $c(0)=N_{hyp}$. Furthermore, for every $s\in (-\varep,\varep)$, $c(s)$ has finite Busemann volume and $c(s)$ corresponds to a strictly convex structure if and only if $s=0$.
\end{corollary}

When $s\neq 0$ we can regard $\rho_s$ as the holonomy of a singular projective structure on $S^3$ with ``cone singularities.'' The singular locus of this structure is the figure-eight knot sitting inside $S^3$ (See \cite{DancigerThesis} for definitions). The properly convex projective structure on the figure-eight knot complement can be recovered from this singular projective structure by deleting the singular locus.

In \cite{Marquis12b} Marquis shows that there exist non-compact hyperbolic 3-manifolds that admit finite volume convex projective deformations. However, his examples are based on the ``bending'' construction of Johnson and Millson \cite{JohnsonMillson87}, which requires the presence of an embedded totally geodesic surface. Since the figure-eight knot complement does not contain any embedded totally geodesic surfaces, we see that the deformations in Theorem \ref{t:maintheorem} are not covered by Marquis' work. Furthermore, to the best of our knowledge these are the first examples of finite volume, properly convex deformations of a non-compact $3$-manifold where the holonomy is explicitly computed. 

The outline of the paper is as follows. Section \ref{s:background} provides background in projective geometry and properly convex projective structures. Section \ref{s:Z^2} defines and discusses certain Lie subgroups of $\PGL[4]{\R}$ that preserves a properly (but not strictly) convex domains. In Section \ref{s:obstruction} we show that the cusp shape of a finite volume hyperbolic 3-manifold places restrictions on deformations that can occur. Section \ref{s:volume} examines the volume of cusps. In Section \ref{s:fig8} we use results from the previous sections and \cite{CooperLongTillman13} to prove Theorem \ref{t:maintheorem}.

\section*{acknowledgements}
I would like to thank Daryl Cooper for first suggesting that the representations in \cite{Ballas12a} might correspond to properly convex projective structures and for several helpful conversations throughout this project. I would also like to acknowledge support from U.S. National Science Foundation grants DMS 1107452, 1107263, 1107367 ``RNMS: GEometric structures And Representation
varieties'' (the GEAR Network). I would also like to thank the referee for several helpful comments and suggestions including a simplified proof of Lemma \ref{l:liealgconjugacy} and for pointing out an incorrect corollary in a previous version.

\section{Background}\label{s:background}

\subsection{Convex Projective Geometry}

Let $\RP^n$ be $n$-dimensional real projective space, that is the quotient of the space $\R^{n+1}\bs \{0\}$ by the natural $\R^\times$ action given by scaling. Let $v\in \R^{n+1}\bs \{0\}$, then we denote by $[v]$ the image of $v$ under the natural quotient map. Let $\PGL[n+1]{\R}$ be the quotient of $\textrm{GL}_{n+1}(\R)$ by its center. It is easy to see that $\PGL[n+1]{\R}$ coincides with the set of self maps of $\RP^n$ induced by linear maps on $\R^{n+1}$. 

The image of a 2-dimensional vector subspace under the quotient map is called a \emph{projective line} and the image of an $n$-plane in $\RP^n$ is called a \emph{projective hyperplane}. The complement of a projective hyperplane in $\RP^n$ is called an \emph{affine patch}. This name comes from the fact that if $A$ is an affine patch, then after a projective coordinate change we can write $A$ in homogeneous coordinates as 
$$A=\{[x_1:\ldots:x_n:1]\mid (x_1,\ldots,x_n)\in \R^n\}.$$

An open subset $\Omega$ of $\RP^n$ is called \emph{convex} if it is contained in some affine patch (i.e. is disjoint from a projective hyperplane) and its intersection with every projective line is connected. If in addition its closure $\ol{\Omega}$ is contained in an affine patch then we say that $\Omega$ is \emph{properly convex}. An alternative characterization of proper convexity is that $\Omega$ does not contain a complete affine line. A point $p\in \partial \Omega$ is strictly convex if $p$ is not contained in the interior of an any line segment in $\partial \Omega$. If $\Omega$ is properly convex and every $p\in \partial \Omega$ is strictly convex then we say that $\Omega$ is \emph{strictly convex}. If $\Omega$ is a properly convex domain and $p\in \partial \Omega$ then $p$ is contained in a (possibly non-unique) hyperplane which is disjoint from $\Omega$. Such a plane is called a \emph{supporting hyperplane}. When $p$ is a $C^1$ point of $\partial \Omega$ there is a unique supporting hyperplane 
containing $p$ which can be identified 
with the tangent plane to $\partial\Omega$ at $p$.   

The space $\RP^n$ is double covered by the sphere $S^n$, which we realize as the quotient of $\R^{n+1}\bs\{0\}$ by the positive scalars. Let $\pi:S^n\to \RP^n$ be the covering map. The automorphisms of $S^n$ are identified with $\slpm[n+1]{\R}$, which
consists of linear transformations with determinant $\pm 1$.  Let $[T]\in \PGL[n+1]{\R}$ be an equivalence class of linear transformations. By scaling $T$ we can arrange that $T\in \slpm[n+1]{\R}$. Additionally, we see that $T\in \slpm[n+1]{\R}$ if and only if $-T\in \slpm[n+1]{\R}$.  As a result, there is a 2-to-1 map, which by abuse of notation we also call $\pi:\slpm[n+1]{\R} \to \PGL[n+1]{\R}$ given by $\pi(T)= [T]$.  If we let $\slpm{\Omega}$ and $\PGL{\Omega}$ be subsets of
$\slpm[n+1]{\R}$ and $\PGL[n+1]{\R}$ preserving $\pi^{-1}(\Omega)$ and $\Omega$, respectively then we see that $\pi$ restricts to a 2-to-1 map from $\slpm[]{\Omega}$ to $\PGL[]{\Omega}$.

When $\Omega$ is properly convex we can construct a section of $\pi$ that is a homomorphism as follows. Since $\Omega\subset \RP^n$ is properly convex, the preimage of $\Omega$ under $\pi$ will consist of two connected components. Every element of $\slpm[]{\Omega$} either preserves both of these components or interchanges them. Furthermore, $T\in \slpm[]{\Omega}$ preserves both components if and only if $-T$ interchanges them. As a result we see that if $[T]\in \PGL[]{\Omega}$ then there is a unique lift of $[T]$ to $\slpm[]{\Omega}$ that preserves both components of $\pi^{-1}(\Omega)$. Mapping $[T]$ to this lift yields the desired section. By using this section we are able to identify $\PGL[]{\Omega}$ with a subgroup of $\slpm[]{\Omega}$. As a result we will regard elements of $\PGL[]{\Omega}$ as linear transformations when convenient.

\subsection{Parabolic Coordinates}

The upper half space model of hyperbolic space offers a coordinate system that allows us to view hyperbolic space ``from a point at infinity,'' thus allowing us to simplify many arguments. In \cite{CooperLongTillman11} a projective analogue of these coordinates is introduced, and in this section we describe a generalization of these coordinates. 

\begin{figure}
 \begin{center}
  \includegraphics[scale=.5]{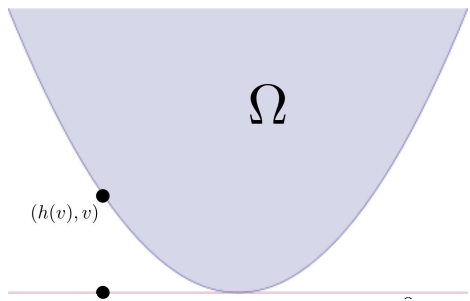}
  \caption{Parabolic coordinates for $\Omega$ \label{f:paraboliccoords}}
 \end{center}

\end{figure}

Let $\Omega$ be a properly convex domain, $p$ a point in $\partial \Omega$ and $H$ a supporting hyperplane containing $p$. The affine patch $\RP^n\bs H$ can be identified with $\R^n$ in such a way that lines through $p$ not contained in $H$ are parallel to the $x_1$ axis. This is done by applying a projective coordinate change that maps $p$ to $[e_1]$ and $H$ to the projective plane dual to $[e_{n+1}]$, where $e_i$ is the $i$th standard basis vector of $\R^{n+1}$. We call the $x_1$ direction the \emph{vertical direction}. A set of coordinates with this property is referred to as \emph{parabolic coordinates centered at $(H,p)$}, or just parabolic coordinates when $H$ and $p$ are clear from context (see Figure \ref{f:paraboliccoords}). It is common to choose parabolic coordinates such that the origin is taken to be a point in $\partial \Omega$ and the plane $H_0$ of points where $x_1=0$ is a supporting hyperplane to $\Omega$, but we will not insist that this is the case.   

Assume we have chosen parabolic coordinates centered at $(H,p)$. If $\gamma\in \PGL{\Omega}$ preserves both $p$ and $H$, then in these coordinates $\gamma$ will act as an affine map. Furthermore, $\gamma$ preserves the foliation of $\R^n$ by lines parallel to the $x_1$-axis. The space of these lines can be identified with $H_0$ via vertical projection and thus induces an affine action of $\gamma$ on $H_0$.

If $p$ is not contained in a line segment in $\partial \Omega\bs H$ then in these coordinates $\partial \Omega\bs H$ can be realized as the graph of a continuous convex function $h:U\to \R$, where $U$ is an open convex subset of $H_0$. Additionally, $\Omega\bs H$ can be identified with the epigraph of this function. For this reason we call $h$ a \emph{boundary function} of $\Omega$. Since $\partial \Omega$ is preserved by $\gamma$ we see that $h$ has certain equivariance properties. Specifically, let $q\in \partial \Omega\bs H$. Then there exists $v\in U$ such that $q=(h(v),v)$ and we see that $\gamma\cdot(h(v),v)=(h(\gamma\cdot v),\gamma\cdot v)$, where the action of $\gamma$ on $v$ is given by the affine action on $H_0$ described in the previous paragraph. 

Parabolic coordinates allow us to define algebraic horospheres as follows. Let $t>0$ and let $e_1$ be the $1$st standard basis vector. Define $\mathcal{S}_t$ as the translation of the portion of $\partial \Omega$ that does not contain any line segments through $p$ by the vector $t e_1$. These sets are called \emph{algebraic horospheres centered at $(p,H)$} (see \cite{CooperLongTillman11} for more details about algebraic horospheres.) Furthermore, if $T>0$ we define an \emph{algebraic horoball centered at $(p,H)$} to be $\bigcup_{t>T}\mathcal{S}_t$.

\subsection{Convex Projective Manifolds}

Let $M$ be an  $n$-manifold. A \emph{projective atlas} on $M$ is a collection of charts, $\phi_\alpha:U_\alpha\to \RP^n$, that cover $M$ with the property that if $U_\alpha$ and $U_\beta$ are charts with nontrivial intersection then $\phi_\alpha\circ \phi_\beta^{-1}$ is locally the restriction of an element of $\PGL[n+1]{\R}$. It can easily be shown that every projective atlas is contained in a unique maximal projective atlas and we call a maximal projective atlas on $M$ a \emph{projective structure} on $M$. A manifold equipped with a projective structure is called a \emph{projective manifold}. From this definition it is clear that a projective manifold is also a smooth manifold. A projective structure is a specific instance of a $(G,X)$ structure on $M$ (see \cite{Ratcliffe06} for an introduction to $(G,X)$ structures). 

If $M$ and $M'$ are projective manifolds of the same dimension, then a continuous map $f:M\to M'$ is \emph{projective} if for each chart $\phi:U\to \RP^n$ of $M$ and each chart $\psi:V\to \RP^n$ of $M'$ such $U$ and $f^{-1}(V)$ intersect the map 
$$\psi\circ f\circ \phi^{-1}:\phi(U\cap f^{-1}(V))\to \psi(f(U)\cap V)$$
is locally the restriction of a element of $\PGL[n+1]{\R}$. Such a map is necessarily smooth. 

The local data of a projective structure can be replaced with more global data as follows. By performing analytic continuation of a chosen chart, we obtain a local diffeomorphism $D:\tilde M\to \RP^n$, where $\tilde M$ is the universal cover of $M$, called a \emph{developing map}. We can also construct a representation $\rho:\fund{M}\to \PGL[n+1]{\R}$  called the \emph{holonomy}. The holonomy representation is equivariant with respect to the developing map in the following way: for $\gamma\in \fund{M}$ and $x\in \tilde M$ 
$$D(\gamma x)=\rho(\gamma)D(x),$$ 
where $\gamma$ acts on $\tilde M$ by deck transformation and $\rho(\gamma)$ acts on $\RP^n$ by projective automorphism. The pair $(D,\rho)$ is called a \emph{developing/holonomy pair}.  While a projective structure does not uniquely determine the pair $(D,\rho)$, but if $(D',\rho')$ is another developing/holonomy pair for the same structure then there exists $g\in \PGL[n+1]{\R}$ such that $D'=gD$ and $\rho'=g\rho g^{-1}$. 

Suppose that we are given a projective structure on $M$ with developing/holonomy pair $(D,\rho)$. If $D$ is a diffeomorphism onto a convex subset, $\Omega$, of $\RP^n$ then we say that the projective structure is \emph{convex}. If in addition, $\Omega$ is properly (resp.\ strictly) convex then we say that the structure is \emph{properly (resp.\ strictly) convex}.  When a projective structure is convex $\rho$ is a discrete and faithful representation and $M$ can be identified with $\Omega/\rho(\fund{M})$. 

A fixed manifold often admits many different projective structures and we would like a space that coherently organizes these structures. Let $N$ be a manifold and suppose that $N$ is either closed or is the interior of a compact manifold with boundary. A \emph{marked projective structure} on $N$ is a pair $(M,f)$, where $M$ is a projective manifold along with a diffeomorphism $f:N\to M$ called a \emph{marking}. Two markings $(M,f)$ and $(M',f')$ are \emph{isotopic} if there exists a projective bijection, $h$, \emph{that is defined on the complement of a collar neighborhood of $\partial M$ onto the complement of a collar neighborhood of $\partial M'$} such that the following diagram commutes up to isotopy. 

$$
\xymatrix{
& M\ar[dd]^h\\
N\ar[ur]^f\ar[dr]_{f'} &\\
& M'
}
$$

Let $\RP(N)$ be the set of isotopy classes of marked projective structures on $N$.  The space $\RP(N)$ can be realized as the quotient of the space of isotopy classes of developing maps by the action of $\PGL[n+1]{\R}$ by post composition. We can thus topologize $\RP(N)$ using the smooth compact open topology on the set of functions from $\tilde N$ to $\RP^n$. 

Let $\rvar[{\PGL[n+1]{\R}}]{\fund{N}}$ be the set of representations from $\fund{N}$ into $\PGL[n+1]{\R}$. We topologize $\rvar[{\PGL[n+1]{\R}}]{\fund{N}}$ using the compact open topology (this coincides with the topology of pointwise convergence on the images of a fixed generating set). Let $\charvar[{\PGL[n+1]{\R}}]{\fund{N}}$ be the quotient of $\rvar[{\PGL[n+1]{\R}}]{\fund{N}}$ by the action of $\PGL[n+1]{\R}$ by conjugation. We topologize $\charvar[{\PGL[n+1]{\R}}]{\fund{N}}$ using the quotient topology.

We can define a map 
\begin{equation}\label{e:holonomy}
hol:\RP(N)\to \charvar[{\PGL[n+1]{\R}}]{\fund{N}}
\end{equation}
 by $hol([(M,f)])=[\rho_M\circ f_\ast]$, where $\rho_M$ is a holonomy for the projective manifold $M$. The following theorem was originally stated by Thurston \cite[Sect.\ 5.1]{ThurstonNotes} and detailed proofs can be found in \cite{Goldman88b} and \cite{CanaryEpsteinGreen87}.
 
\begin{theorem}[Thurston \cite{ThurstonNotes}]\label{t:thurstonholonomy}
Let $N$ be the interior of a compact smooth manifold. Then $hol:\RP(N)\to\charvar[{\PGL[n+1]{\R}}]{\fund{N}}$ is a local homeomorphism. 
\end{theorem}
As a consequence we see that elements of $\RP(N)$ can be locally parameterized by $\charvar[{\PGL[n+1]{\R}}]{\fund{N}}$. Let $\gold{N}$ be the set of isotopy classes that contain a properly convex representative. That is $[(M,f)]\in \gold{N}$ if there is $(M',f')\in [(M,f)]$ such that $M'$ is a properly convex projective manifold. The behavior of $\gold{N}$ as a subset of $\RP(N)$ depends on whether or not $N$ is closed. When $N$ is closed, work of Koszul \cite{Koszul68} shows that $\gold{N}$ is an open subset of $\RP(N)$. Furthermore, Benoist has shown in \cite{Benoist05} that when $N$ is closed $\gold{N}$ is a closed subset of $\RP(N)$. However, when $N$ is non-compact there exist, in general, sequences of non-discrete representations from $\fund{N}$ into $\PGL[n+1]{\R}$ that converge to the holonomy of a properly convex projective structure on $N$. As a result we see that $\gold{N}$ is not generally an open subset of $\RP(N)$.  

However, recent work of Cooper, Long, and Tillmann has given a sufficient condition for a small deformation of the holonomy of a properly convex projective structure to continue to be the holonomy of a properly convex projective structure. Their condition is phrased in terms of the ends of the manifold having the structure of \emph{generalized cusps}, which are generalizations of cusps of hyperbolic manifolds and are defined in Section \ref{s:Z^2}. Specifically, they prove the following theorem.

\begin{theorem}[{\cite[Thm 0.1]{CooperLongTillman13}}]\label{t:holonomytheorem}
Let $M$ be a properly convex projective manifold with strictly convex boundary and holonomy $\rho$. Suppose that $M=A\cup \mathcal{B}$ where $A$ is compact and $\partial A=A\cap \mathcal{B}=\partial \mathcal{B}$, and each component of $B$ of $\mathcal{B}$ is a generalized cusp. If $\rho'$ is sufficiently close to $\rho$ in $\rvar[{\PGL[n+1]{\R}}]{\fund{M}}$ and for each component $B\subset \mathcal{B}$ there is a convex projective structure on $B$ which is a generalized cusp with holonomy $\rho'\vert_{\fund{B}}$ then there is a properly convex projective structure on $M$ with holonomy $\rho'$. 
\end{theorem}

\subsection{Hilbert Metric and Busemann Volume}

Every properly convex subset of $\RP^n$ comes equipped with a \emph{Hilbert metric} that is invariant under $\PGL{\Omega}$ and is defined as follows: Let $\Omega$ be an open properly convex domain in $\RP^n$ and let $x,y$ be distinct points in $\Omega$. Since $\Omega$ is properly convex, it follows that the line segment connecting $x$ to $y$ intersects $\partial \Omega$ in exactly two points, $a$ and $b$, such that $a$ is closer to $x$ and $b$ is closer to $y$ (see Figure \ref{f:hilbertmetric}). We define the Hilbert distance between $x$ and $y$ as $\dhilb{x}{y}:=\log\left([a:x:y:b]\right)$. Here $[a:x:y:b]=\frac{\abs{y-a}\abs{x-b}}{\abs{y-b}\abs{x-a}}$ is the projective cross ratio and $\abs{\cdot}$ is the Euclidean norm on any affine patch containing $\Omega$. When $x=y$ we define $\dhilb{x}{y}=0$. A proof that $d_\Omega$ is a metric can be found in \cite{delaHarpe93}. Since the cross ratio is invariant under projective automorphisms it is clear that the Hilbert metric is invariant under $\PGL{\Omega}$. 
When $\Omega$ is an ellipsoid the Hilbert metric coincides with twice the hyperbolic metric on the Klein model of hyperbolic space.

\begin{figure}
 \begin{center}
  \includegraphics[scale=.6]{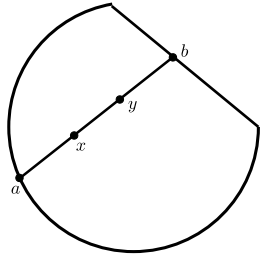}
  \caption{The line segment defining the Hilbert metric \label{f:hilbertmetric}}
 \end{center}

\end{figure}

We now define a $\PGL{\Omega}$-invariant measure on $\Omega$. We first recall the definition of \emph{Hausdorff measure} for an arbitrary metric space. Let $(X,d)$ be a metric space and let $r$ be a non-negative real number. Let $E\subset X$ and let $S^\varep_E$ be the set of all countable covers of $E$ by sets of diameter less than $\varep$. We define 
$$m_\varep^r(E)=\inf_{\{E_i\}\in S^\varep_E}c_r\sum_{i}diam(E_i)^r,$$
for some constant $c_r$, and we let 
$m^r(E)=\lim_{\varep \to 0}m_\varep^r(E).$ This construction defines an outer measure on $X$ which when restricted to the $\sigma$-algebra of Borel sets gives the \emph{$r$-dimensional Hausdorff measure on $X$}. 

 We can use the Hilbert metric on $\Omega$ to define the $n$-dimensional Hausdorff measure on Borel subsets of $\Omega$. We normalize so that $c_n=\alpha_n$, where $\alpha_n$ is the Lebesgue measure of the Euclidean ball of diameter 1. We denote this measure by $\mu_{\Omega}$ and call it the \emph{Busemann volume on $\Omega$}. Since this measure is defined using only the Hilbert metric it is easy to see that it is invariant under $\PGL{\Omega}$. 
 
 The above definition of Busemann volume is not conducive to performing computations and so we now recast the definition in a way that is more computationally convenient (compare Marquis \cite{Marquis13}). Let $\mathcal{A}$ be an affine patch containing $\Omega$ and we equip $\mathcal{A}$ with the Euclidean metric $\abs{\cdot}$. We can use the Hilbert metric to define a Finsler structure on $\Omega$ as follows. We can identify $T_x\Omega$ with $\mathcal{A}$ and if $v\in T_x\Omega$ we define 
\begin{equation}\label{e:hilbertnorm}
\Abs{v}_\Omega=\left.\frac{d}{dt}\right|_{t=0}\dhilb{x}{x+tv}=\abs{v}\left(\frac{1}{\abs{x-p_{-}}}+\frac{1}{\abs{x-p_{+}}}\right),
\end{equation}
where $p_{+}$ and $p_{-}$ are the intersection points of $\partial \Omega$ and the line in $\Omega$ through $x$ determined by $v$. Using this norm we can define the $n$-dimensional Hausdorff measure on $T_x\Omega$. We normalize this measure by choosing the constant $c_n=\alpha_n$, and denote it by $\mu^x_{\Omega}$. 

In order to facilitate computations we define another Finsler structure on $\Omega$. In this Finsler structure the norm on each tangent space is given by the Euclidean norm on $\mathcal{A}$. Similarly, we can define a measure on each tangent space using the Lebesgue measure, $\mu_L$. The identity map on $\Omega$ gives a map between these two Finsler manifolds. Let $B\subset \Omega$ be a Borel set, then by the ``change of variables formula'' \cite[Prop 5.5.11]{BuragoBuragoIvanov01} we see that there exists $f:\Omega\to \R$ such that 
\begin{equation}\label{e:volform1}
\hilbvol{B}=\int_Bf(x)d\mu_L(x),
\end{equation}
furthermore, the function $f$ can be easily described. The measures $\mu_L$ and $\mu^x_\Omega$ are both translation invariant measures and therefore differ by multiplication by a constant. The function $f(x)$ is this constant. Let $B_x^\Omega(1)$ be the unit ball in $T_x\Omega$ for the norm $\Abs{\cdot}_{\Omega}$. A simple computation shows that $f(x)=\frac{\mu^x_\Omega(B_x^\Omega(1))}{\mu_L(B_x^\Omega(1))}$. By work of Busemann \cite{Busemann47} we see that $\mu^x_\Omega(B_x^\Omega(1))=\alpha_n$, and we see that \eqref{e:volform1} becomes 
\begin{equation}\label{e:hilbertvolume}
\hilbvol{B}=\int_B\frac{\alpha_n}{\mu_L(B_x^\Omega(1))}d\mu_L(x).
\end{equation}

The Busemann volume is $\PGL{\Omega}$-invariant and thus descends to a Borel measure on any properly convex manifold $M=\Omega/\Gamma$, where $\Gamma$ is a discrete torsion-free subgroup of $\PGL{\Omega}$. We refer to this measure as the \emph{Busemann volume on $M$} and denote it $\mu_M$.

We close this section by mentioning some comparison properties of the Hilbert metric and Busemann volume that we will use throughout. Let $\Omega\subset \Omega'$ be two properly convex domains. If $x,y\in \Omega$ then a simple computation shows that $\dhilb[\Omega']{x}{y}\leq \dhilb[\Omega]{x}{y}$. Similarly, if $A\subset \Omega$ is a Borel subset then $\hilbvol[\Omega']{A}\leq \hilbvol[\Omega]{A}$.

\section{Properly Convex Cusps and Representations of $\Z^2$}\label{s:Z^2}

In this section we construct properly convex cusps that can be realized as small deformations of standard cusps coming from the complete hyperbolic structure on a cusped 3-manifold. To begin with let $\mathfrak{L}_0$ be the Abelian Lie subalgebra of $\mathfrak{gl}_4$ consisting of matrices of the form 
$$\begin{pmatrix}
   0 & r & s & 0\\
   0 & 0 & 0 & r\\
   0 & 0 & 0 & s\\
   0 & 0 & 0 & 0
  \end{pmatrix}
$$
If we exponentiate $\mathfrak{L}_0$ we get a Lie subgroup $L_0$ of $\GL[4]{\R}$ consisting of matrices of the form  
$$\begin{pmatrix}
   1 & r & s & \frac{1}{2}(r^2+s^2)\\
   0 & 1 & 0 & r\\
   0 & 0 & 1 & s\\
   0 & 0 & 0 & 1
  \end{pmatrix}
$$
The elements of $L_0$ preserve the 3-dimensional affine subspace ${\{(x_1,x_2,x_3,x_4)\in \R^4\vert x_4=1\}}$ and can thus be regarded as affine transformations of $\R^3$. Furthermore, by projectivizing we get an embedded copy of $L_0$ in $\pgl$, which we also call $L_0$.

Let $D_0=\{(x_1,x_2,x_3)\vert x_1>\frac{1}{2}(x_2^2+x_3^2)\}\subset \R^3$, then we see that the affine action of $L_0$ preserves $D_0$. Furthermore, if we let $c\geq 0$ and let $\mathcal{H}^0_c$ be the $L_0$-orbit of $(c,0,0)$ then we get a foliation of $D_0$. Explicitly, $\mathcal{H}^0_c$ is the graph of the strictly convex function $f_c(x_2,x_3)=\frac{1}{2}(x_2^2+x_3^2)+c$ and we can think of $D_0$ as the epigraph of $f_0$. We can now realize $D_0$ as a subset of $\RP^3$ via the embedding $(x_1,x_2,x_3)\mapsto[x_1:x_2:x_3:1]$ of $\R^3$ into $\RP^3$. In this way we realize parabolic coordinates for $D_0$ in which $f_0$ is a boundary function. It is easy to see that $D_0$ is the copy of $\HH^3$ obtained by projectivizing the negative cone of the quadratic form 

$$Q(x_1,x_2,x_3,x_4)=\frac{1}{2}(x_2^2+x_3^2)-x_1x_4.$$

With this point of view we see that the foliation $\mathcal{H}^0_c$ is just the foliation by horospheres centered at the point $[1:0:0:0]$ (which we think of as the point at $\infty$). 

If we let $\Gamma_0$ be a  lattice inside $L_0$ then $D_0/\Gamma_0$ is a properly (in fact strictly) convex manifold diffeomorphic to $T^2\times(0,\infty)$ via a diffeomorphism that maps $\mathcal{H}^0_c/\Gamma_0$ to $T^2\times \{c\}$. Furthermore, each $T^2\times \{c\}$ is equipped with the Euclidean similarity structure whose developing map is given by the restriction of the projection $[x_1:x_2:x_3:1]\mapsto [x_2:x_3:1]$ to $\mathcal{H}^0_c$. 

We now recall the construction, in dimension 3, of generalized cusps from \cite{CooperLongTillman13}. Let $\Phi_t$ be the affine flow on $\R^3$ given by the affine transformation
$$\Phi_t=\begin{pmatrix}
          1 & 0 & 0 & t\\
          0 & 1 & 0 & 0\\
          0 & 0 & 1 & 0\\
          0 & 0 & 0 & 1
         \end{pmatrix}
$$
Let $M$ be a convex projective $3$-manifold, then there exists a convex open $\Omega\subset \RP^3$ and a discrete group $\Gamma\subset \PGL{\Omega}$ such that $M\cong \Omega/\Gamma$. We say that $M$ is a \emph{complete generalized cusp} if $\Gamma$ is conjugate in $\pgl$ into the centralizer of $\Phi_t$.

A \emph{local hyperplane} in $M$ is an evenly covered subset (with respect to the universal cover $\Omega$) such that some (hence any) lift to $\Omega$ develops into a convex subset of an affine 2-plane in $\R^3$. Let $W$ be 2-dimensional submanifold of $M$. A point $p\in W$ is called \emph{strictly convex} if there exists a local hyperplane $P$ such that $P\cap W=p$. If every point of $W$ is strictly convex then we say that $W$ is \emph{strictly convex}. 

Let $B$ be an open properly convex submanifold of $M$ such that $\partial B$ is strictly convex. Suppose also that there is an identification of $B$ with $\partial B\times [0,\infty)$ such that some (hence any) lift of $\{b\}\times [0,\infty)$ to $\Omega$ develops into a flowline of $\Phi_t$ for every $b\in \partial B$. If $B$ has these properties then $B$ is called \emph{generalized cusp}.  

Since $\Gamma_0$ is contained in the centralizer of $\Phi_t$ we see that $\R^3/\Gamma_0$ is a complete generalized cusp. The properly convex submanifold $D_0/\Gamma_0$ has strictly convex boundary and admits a product structure given by flowlines of $\Phi_t$ and is thus a \emph{generalized cusp}

We now repeat the above construction while allowing ourselves to vary the initial Lie algebra. Let $\mathfrak{L}_t$ be the Abelian Lie subalgebra of $\mathfrak{gl}_4$ consisting of matrices of the form
$$\begin{pmatrix}
   0 & r & s & 0\\
   0 & t r & 0 & r\\
   0 & 0 & 0 & s\\
   0 & 0 & 0 & 0
  \end{pmatrix}
$$
If we exponentiate $\mathfrak{L}_t$ then we obtain the Lie subgroup, $L_t\leq \GL[4]{\R}$ of affine transformations of the form 
$$\begin{pmatrix}
   1 & \frac{e^{t r}-1}{t} & s & \frac{e^{tr}-t r-1}{t^2}+\frac{s^2}{2}\\
   0 & e^{t r} & 0 & \frac{e^{t r}-1}{t}\\
   0 & 0 & 1 & s\\
   0 & 0 & 0 & 1
  \end{pmatrix}
$$
By projectivizing we obtain an embedded copy of $L_t$ inside $\pgl$, which we also call $L_t$.

Using parabolic coordinates, we can again regard the $L_t$-orbit of $(0,0,0)$ as the graph of a strictly convex function and let $D_t$ be the epigraph of this function. As above, we can regard $D_t$ as a domain in $\RP^3$. The following lemma demonstrates some convexity properties of $D_t$ in this context. 
 
\begin{lemma}\label{l:propconvex}
 For each $t>0$, $D_t$ is a properly, but not strictly convex subset of $\RP^3$. 
\end{lemma}
\begin{proof}
 Let $\mathfrak{L}'$ be the Abelian Lie algebra of matrices of the form 
 \begin{equation}
\begin{pmatrix}
    0 & 0 & b & -a\\
    0 & a & 0 & 0\\
    0 & 0 & 0 & b\\
    0 & 0 & 0 & 0
   \end{pmatrix}\label{e:liealg}
\end{equation}
which exponentiates to the Lie group, $L'$ of matrices of the form
\begin{equation}\label{e:liegroup1}
 \begin{pmatrix}
  1 & 0 & b & \frac{1}{2}b^2-a\\
  0 & e^a & 0 & 0\\
  0 & 0 & 1 & b\\
  0 & 0 & 0 & 1
 \end{pmatrix}
\end{equation}
Next, let $D'$ be the epigraph of the $L'$-orbit of the point $(0,1,0)$ (see Figure \ref{f:domain1}).

\begin{figure}
 \begin{center}
  \includegraphics[scale=.7]{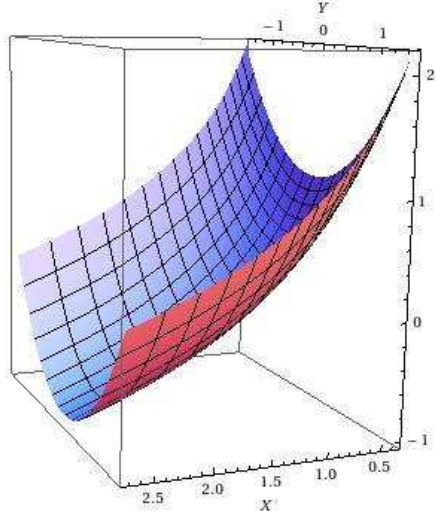}
  \caption{The boundary of $D'$\label{f:domain1}}
 \end{center}
\end{figure}

The affine transformation
\begin{equation}\label{e:Vt}
V_t=\begin{pmatrix}
       1/t^2 & 1/t^2 & 0 & -1/t^2\\
       0 & 1/t & 0 & -1/t\\
       0 & 0 & 1/t & 0\\
       0 & 0 & 0 & 1
      \end{pmatrix}
\end{equation}
conjugates $\mathfrak{L}_t$ to $\mathfrak{L}'$ (under this conjugation, $a=tr$ and $b=ts$). Since $V_t\cdot[0:1:0:1]=[0:0:0:1]$ we see that $V_t^{-1}D_t=D'$. Thus it suffices to show that $D'$ is properly, but not strictly convex. 

To see this observe that the $L'$-orbit of $(0,1,0)$ is the set 
$$\left\{\left(\frac{b^2}{2}-a,e^{a},b\right)\vert a,b\in \R\right\}.$$
Under the coordinate change $A=e^{a}$ and $B=b$, this orbit can be written as 
$$\left\{\left(\frac{B^2}{2}-\log(A),A,B\right)\vert A,B\in \R, A>0\right\},$$
and we see that 
$${D'=\left\{\left(x_1,x_2,x_3\right)\in\R^3\vert x_1>\frac{x_3^2}{2}-\log(x_2),x_2>0\right\}}.$$
Thus we can realize $D'$ as the epigraph of the function $F(x_2,x_3)=\frac{x_3^2}{2}-\log(x_2)$. By computing the Hessian we see that $F$ is a convex function and thus $D'$ is a convex set. The function $F$ is the boundary function for parabolic coordinates for $D'$ centered at $(p,H)$, where  $p=[1:0:0:0]$ and $H$ is the projective plane dual to $[0:0:0:1].$

We next show that $D'$ does not contain a complete affine line. Suppose for contraction that $\ell$ is an affine line that is entirely contained in $D'$. Our argument will make use of the parabolic coordinates introduced in the previous paragraph. If the tangent vector of $\ell$ has any component in the $x_2$ direction then $\ell$ will intersect the $x_1x_3$-plane, which is disjoint from $D'$. Thus $\ell$ must lie in a plane, $P$, parallel to the $x_1x_3$-plane. The intersection of $P$ with $D'$ is the epigraph of a function of the form $x_3^2/2+C$ for some constant $C$. It is easy to see that any affine line in $P$ intersect the graph of $x^2_3/2+C$, which contradicts the fact that $\ell$ is entirely contained in $D'$.   

Finally, we show that $D'$ is not strictly convex by exhibiting an explicit line segment in $\partial D'$. Let $c>0$ and observe that the affine ray 
$$R_c=\{[c u:u:0:1]\mid u>T_c\},$$
(where $T_c$ is a positive constant that depends on $c$) is completely contained in $D'$. The ideal endpoint of this affine ray is $[c:1:0:0]$, which is contained in the boundary of $D'$ in $\RP^3$. Varying $c$ parameterizes a line segment that is contained in $\partial D'$. 
\qed
\end{proof}

The domains $D'$ and $D_t$ also admit foliations by horospheres. Let $\mathcal{H}'_\kappa$ be the $L'$ orbit of $(\kappa,1,0)$. Let $\mathcal{H}^t_\kappa$ be the $L_t$-orbits of the points $(\kappa,0,0)$. The sets $\bigcup_{s>0}\mathcal{H}'_s$ (resp.\ $\bigcup_{s>0}\mathcal{H}^t_s$) are easily seen to be a foliation of $D'$ (resp.\ $D_t$) by algebraic horospheres. For this reason we call $\mathcal{H}'_\kappa$ (resp.\ $\mathcal{H}^t_\kappa$) \emph{$L'$-horospheres} (resp.\ \emph{$L_t$-horospheres}). 

If $\Gamma_t$ is a lattice inside $L_t$ then $D_t/\Gamma_t$ is a properly convex manifold that is again diffeomorphic to $T^2\times(0,\infty)$ by a diffeomorphism that sends $\mathcal{H}^t_\kappa/\Gamma_t$ to $T^2\times\{\kappa\}$. In this case the map $[x_1:x_2:x_3:1]\mapsto[x_2:x_3:1]$ restricted to $\mathcal{H}^t_\kappa$ is the developing map for an affine structure on $\mathcal{H}^t_\kappa/\Gamma_t$. Let $E_t=\{(x_1,x_2,x_3)\in \R^3\vert x_2>-1/t\}$, then $E_t/\Gamma_t$ is a complete generalized cusp and $D_t/\Gamma_t$ is a generalized cusp. As $t\to 0$ we see that $\mathfrak{L}_t$ converges to $\mathfrak{L}_0$, $E_t$ converges to $\R^3$, and  $D_t$ converges to $D_0$ (see Figure \ref{f:domaindeformation}). Furthermore, if $\lim_{t\to 0}\Gamma_t$ converges to a  lattice in $L_0$ then $D_t/\Gamma_t$ 
will converge to a hyperbolic cusp and the affine tori $\mathcal{H}^t_\kappa/\Gamma_t$ will converge to a Euclidean torus (see Figure \ref{f:affinedeformation}). In \cite{Baues11}, Baues describes in more detail how 2-dimensional affine tori can converge to one another. In his notation, the convergence we have just described is a sequence of type $C_1$ affine structures converging to a type $D$ affine structure.       

\begin{figure}[h]
 \begin{center}
  \includegraphics[scale=.4]{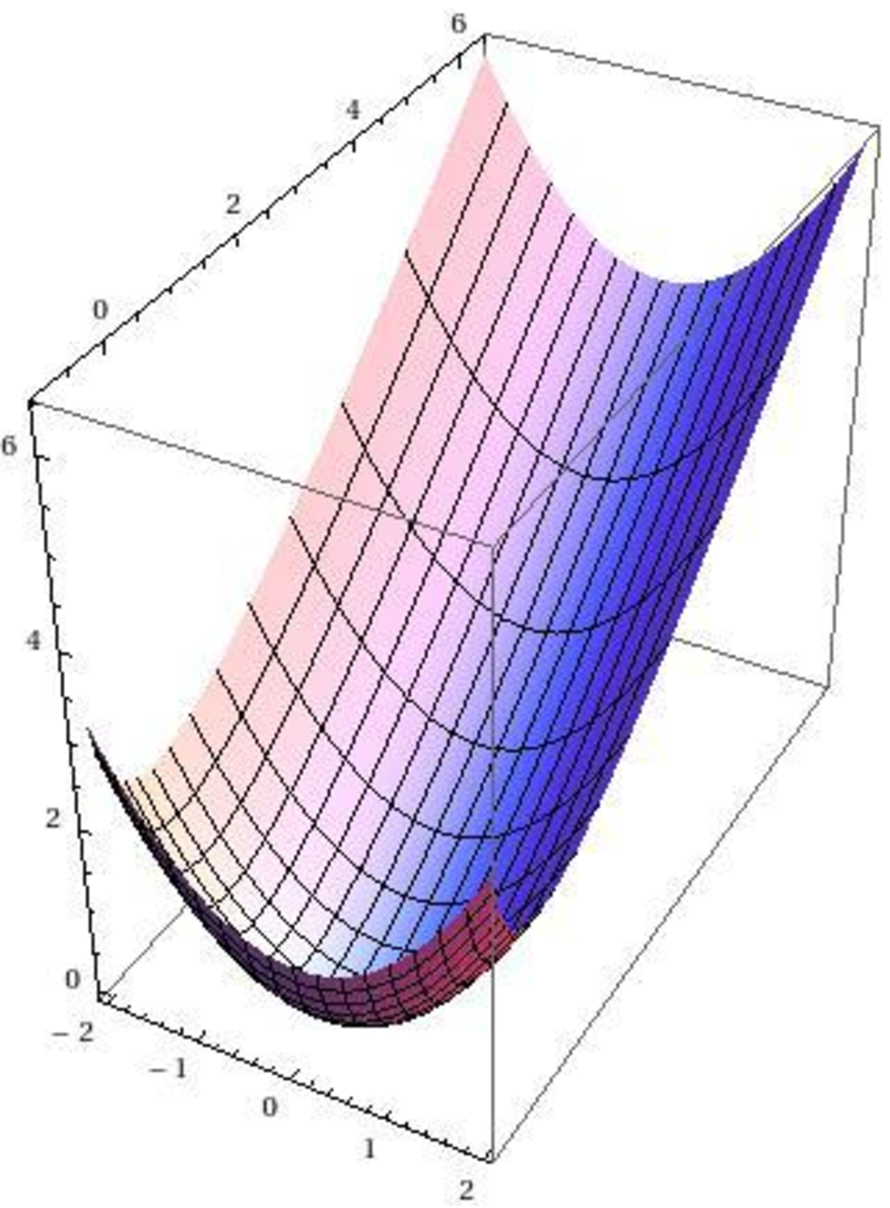}
  \includegraphics[scale=.4]{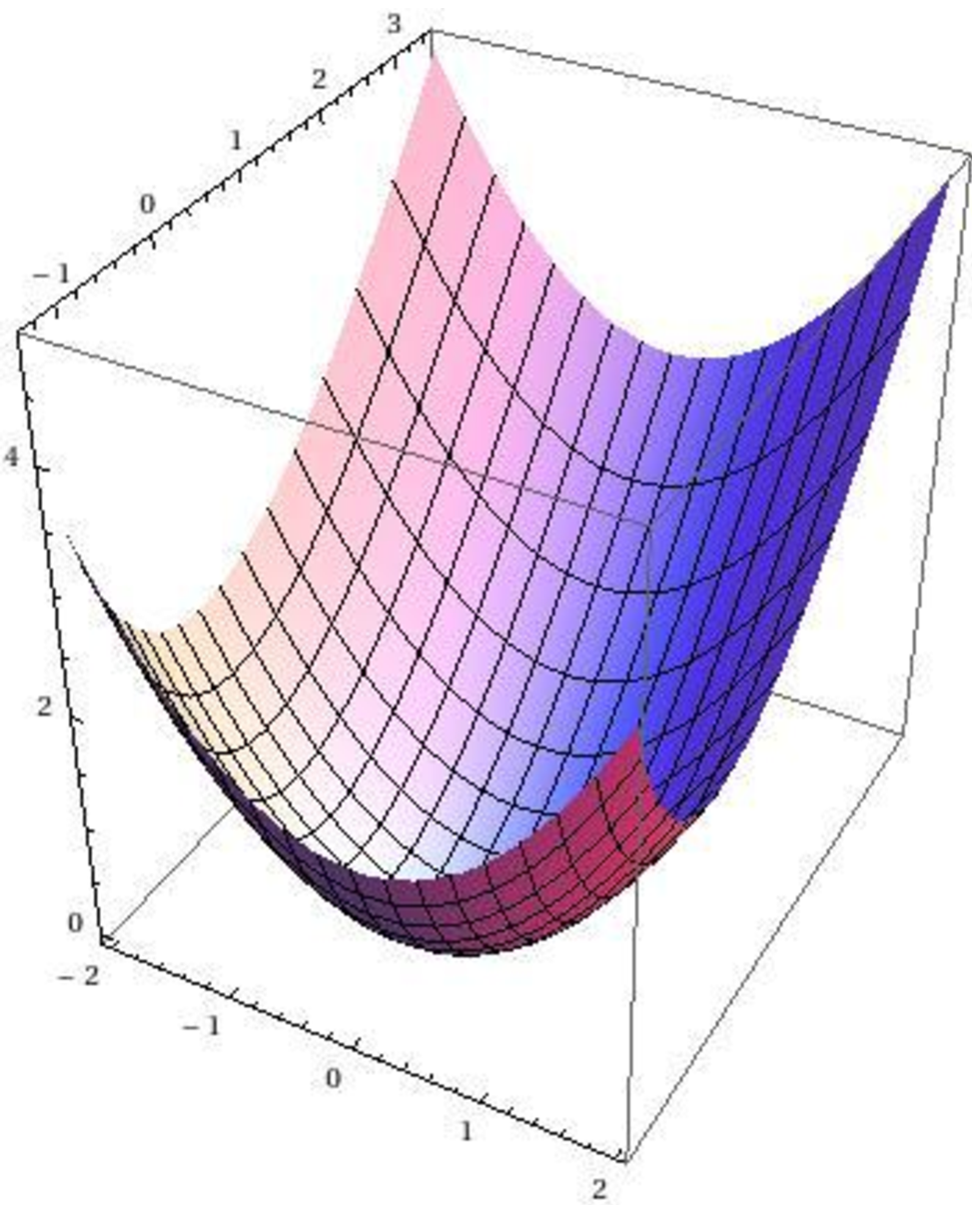}
  \includegraphics[scale=.4]{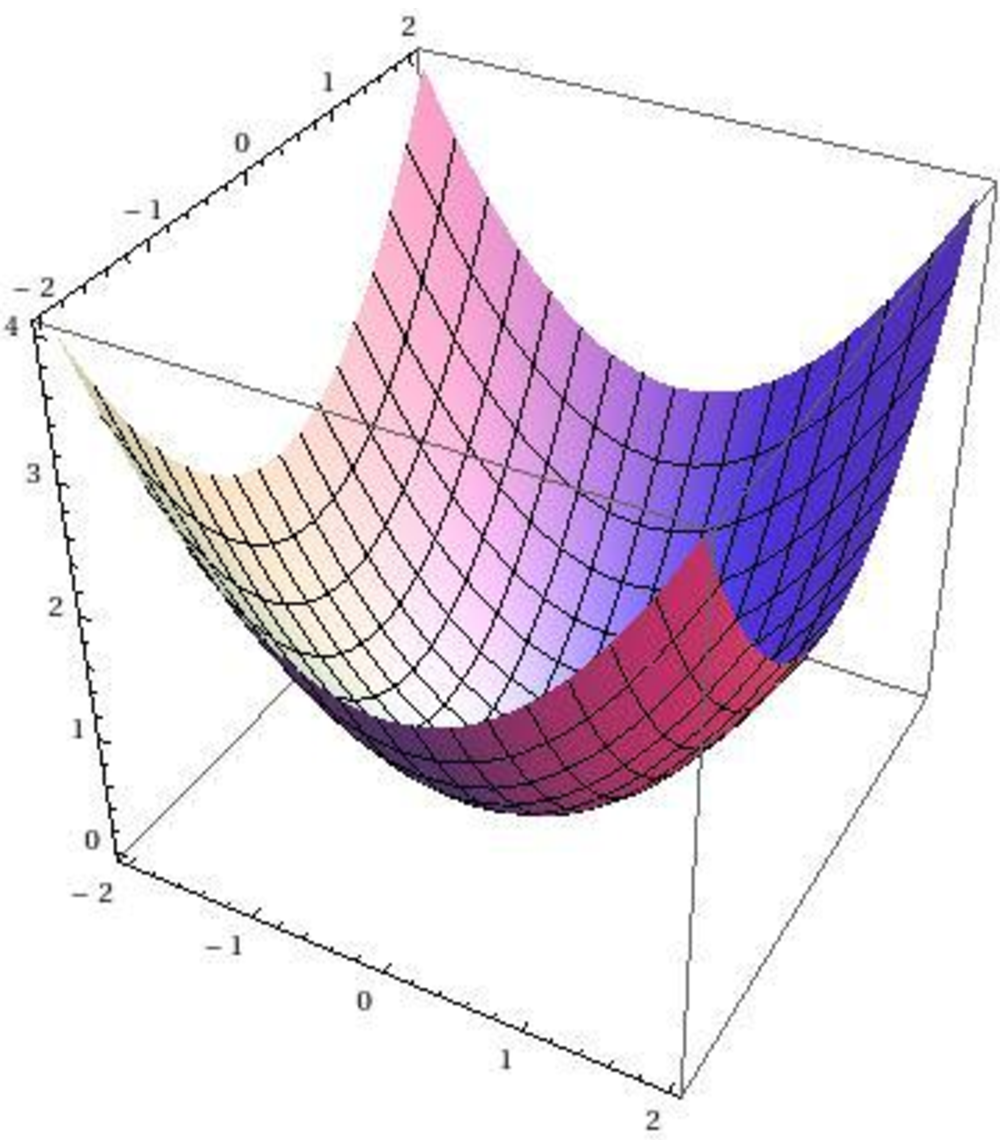}
  \caption{Convergence of $D_t$ to $D_0$.\label{f:domaindeformation}}
 \end{center}
 \end{figure}

\begin{figure}
 \begin{center}
  \includegraphics[scale=.4]{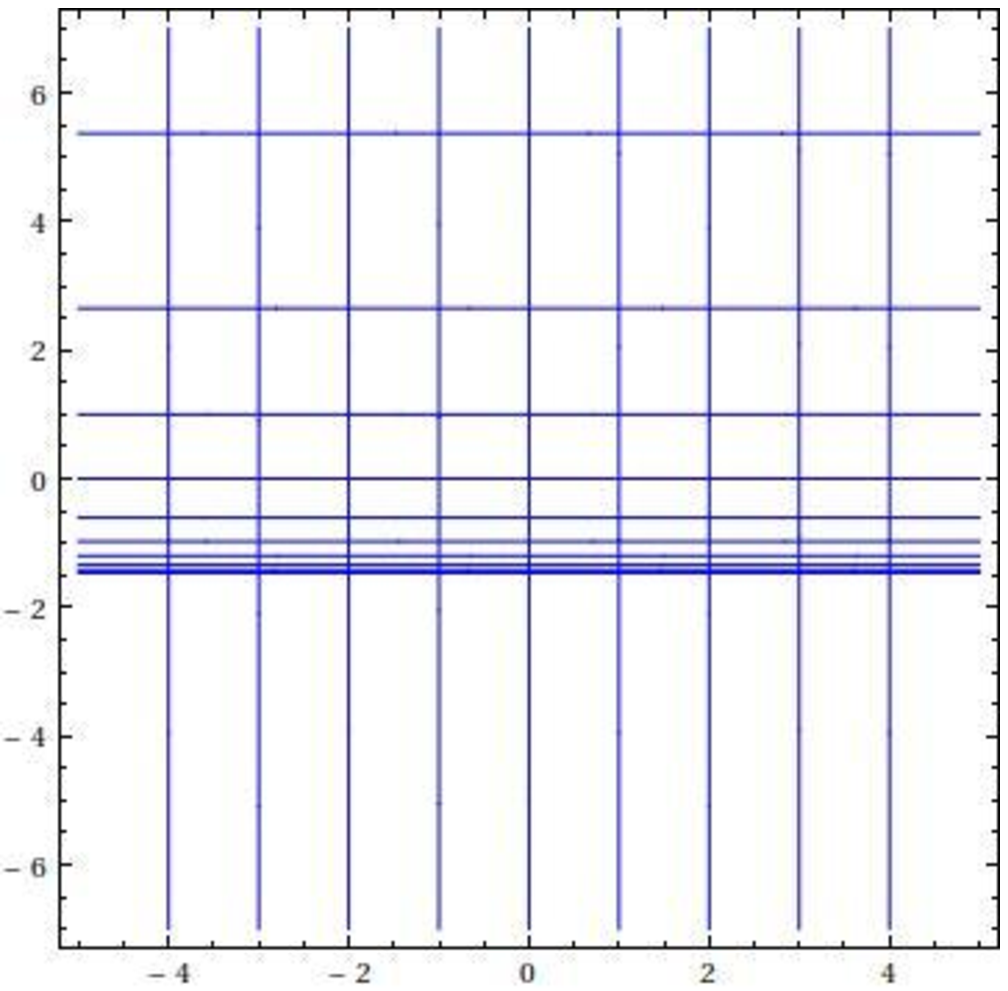}
  \includegraphics[scale=.4]{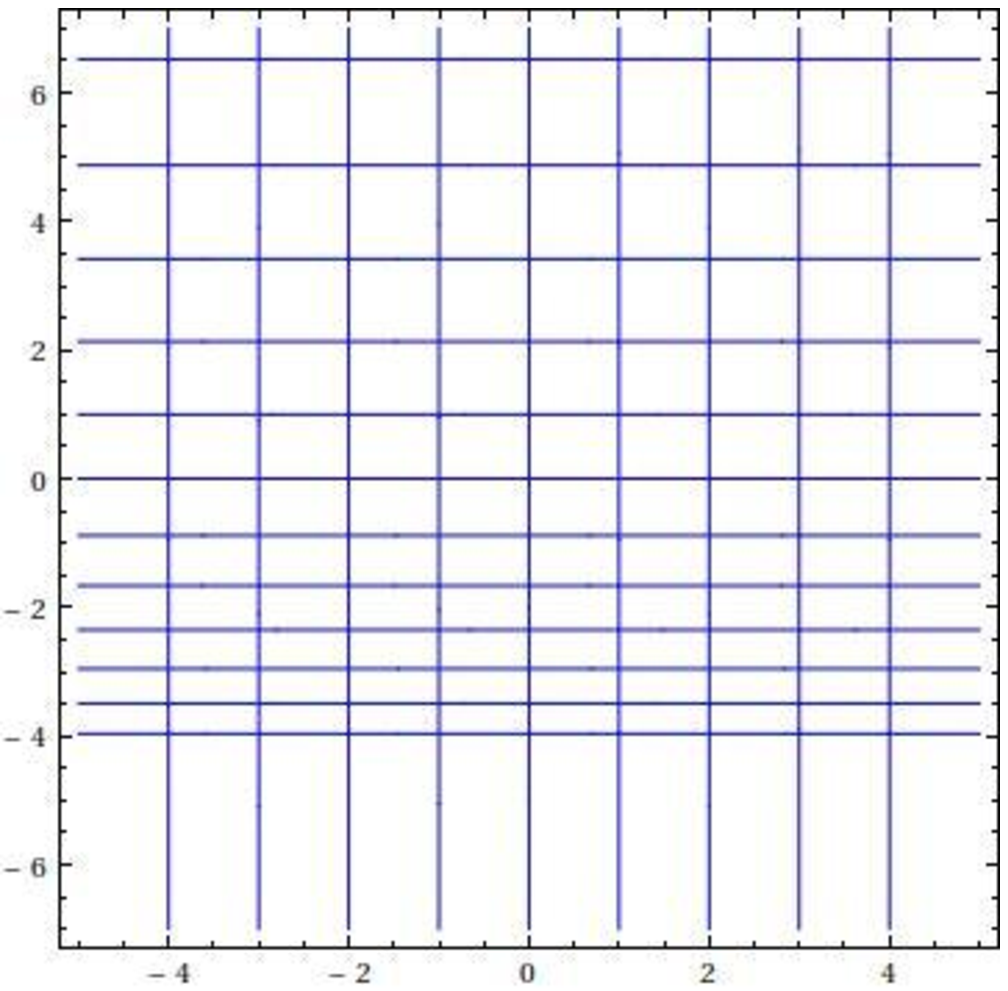}
  \includegraphics[scale=.4]{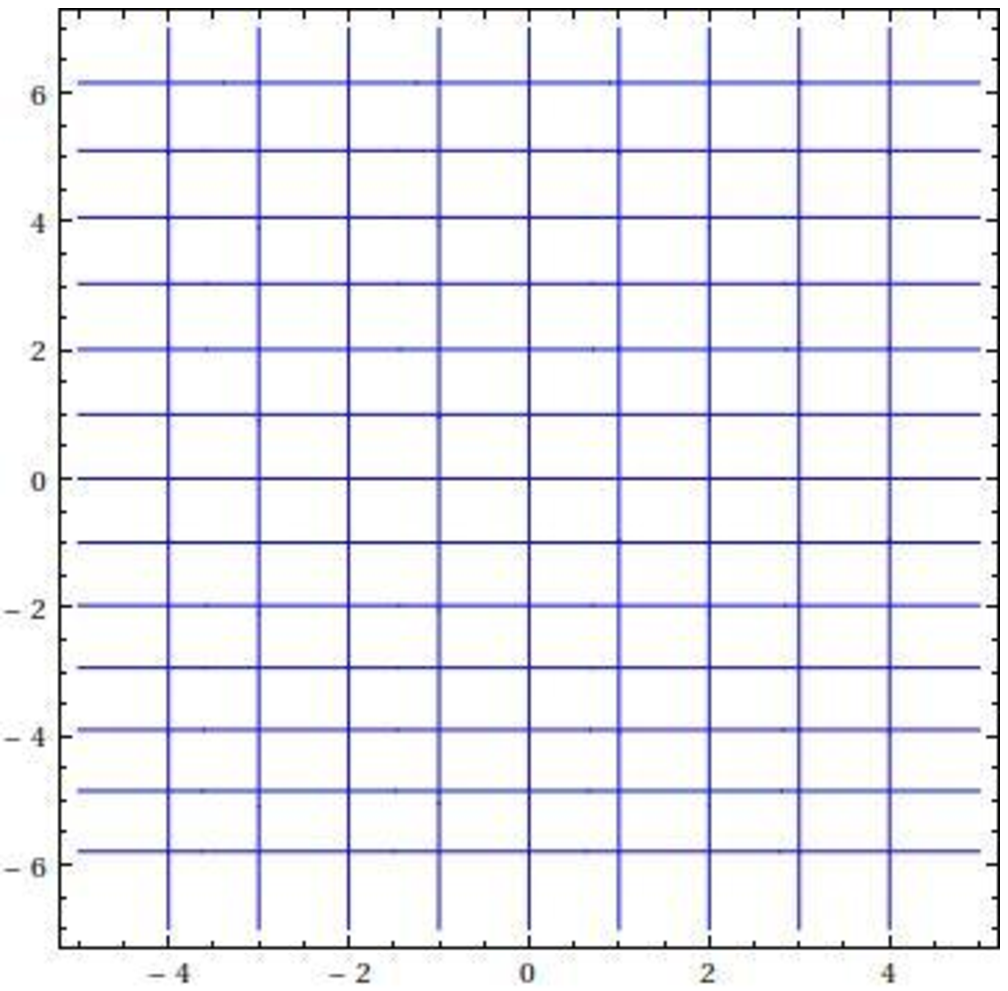}
  \caption{Fundamental domains of affine tori converging to fundamental domains of a Euclidean torus \label{f:affinedeformation}}
 \end{center}

\end{figure}

\subsection{Rank 2 Lattices in $L'$}

To conclude this section we characterize the holonomy of cusps which are projectively equivalent to $D'/\Gamma'$, where $\Gamma'$ is a lattice of $L'$. Additionally, we discuss when a family of such holonomies converges to the holonomy of a hyperbolic cusp. We begin by analyzing which rank 2 discrete abelian subgroups of $\pgl$ are conjugate into $L'$. 

We begin by observing that if $A$ and $B$ are two commuting matrices in $\GL[4]{\R}$ with positive, real eigenvalues then the group $\langle A,B\rangle$ is contained in a Lie subgroup, $L_{AB}$, of $\GL[4]{\R}$ isomorphic to $\R^2$. This Lie group is isomorphic to its Lie algebra, $\mathfrak{L}_{AB}$ via the exponential map. Let $\alpha$ and $\beta$ be the elements of $\mathfrak{L}_{AB}$ corresponding to $A$ and $B$ via the exponential map. We can now define a map $m:\mathfrak{L}_{AB}\to \R[t]$ that assigns to an element of $\mathfrak{L}_{AB}$ its minimal polynomial. We call this function the \emph{minimal polynomial map} for $\mathfrak{L}_{AB}$. It is clear that the minimal polynomial map is invariant under conjugation by an element of $\GL[4]{\R}$. 

Next, we examine how the function $m$ behaves on the Lie algebra $\mathfrak{L'}$. An element of $x\neq 0$ of $\mathfrak{L'}$ has the form 
$$x=\begin{pmatrix}
   0 & 0 & b & -a\\
   0 & a & 0 & 0\\
   0 & 0 & 0 & b\\
   0 & 0 & 0 & 0
  \end{pmatrix}
$$
and we see that $m(x)=t^{n(a,b)}(t-a)$, where 
$$n(a,b)=\left\{\begin{array}{rcl}
                 2 & \rm{if} & ab=0\\
                 3 & &\rm{otherwise} 
                \end{array}
\right.$$
More generally, we see that 
\begin{equation}\label{e:minpolyform}
m(x)=t^{n(x)}(t-f(x))
\end{equation}
 where $f(x)$ is a linear functional on $\mathfrak{L'}$ and $n:\mathfrak{L'}\bs\{0\}\to \{2,3\}$. The kernel of this linear functional is contained in $n^{-1}(2)$, and $n^{-1}(2)$ is a union of two distinct linear subspaces of $\mathfrak{L'}$. Elements of $L'$ that correspond under the exponential map to elements of $\ker f$ are called \emph{pure translations}. Elements of $L'$ that correspond under the exponential map to elements of $n^{-1}(2)\bs \ker f$ are called \emph{pure dilations}. Elements of $L'$ are affine transformations of $\R^3$ that preserve the foliation of $\R^3$ by vertical lines. The set of these lines can be identified with $\R^2$ in such a way that $L'$ acts by affine transformations. The action on $\R^2$ of a pure translation corresponds to an affine transformation whose linear part is the identity and the action on $\R^2$ of a pure dilation corresponds to an affine transformation whose linear part has distinct real eigenvalues and whose translational part is trivial.  
  
 At first, one might hope that any Lie algebra for which the minimal polynomial map has these properties would be conjugate to $\mathfrak{L'}$. However the Lie algebra $\mathfrak{L'_{-}}$ of elements of the form 

$$\begin{pmatrix}
   0 & 0 & b & a\\
   0 & a & 0 & 0\\
   0 & 0 & 0 & b\\
   0 & 0 & 0 & 0
  \end{pmatrix}
$$
gives rise to a minimal polynomial map of the same form, but $\mathfrak{L'}$ is not conjugate to $\mathfrak{L_{-}'}.$

\begin{remark}\label{r:convexorbit}
 To see why $\mathfrak{L'}$ and $\mathfrak{L'_{-}}$ are not conjugate we look at the corresponding Lie groups. As we have seen, the Lie group $L'$ corresponding to $\mathfrak{L'}$ preserves an open, properly convex domain bounded by the closure of the $L'$-orbit of $[0:1:0:1]$. The Lie group $L'_{-}$ corresponding to $\mathfrak{L'_{-}}$ preserves no such properly convex domain. 
\end{remark}

However, the following lemma shows that once we know that our minimal polynomial map has these properties that this is the only ambiguity.

\begin{lemma}\label{l:liealgconjugacy}
Let $\mathfrak{F}$ be a two dimensional Abelian Lie subalgebra of $\mathfrak{gl}_4$. Suppose that for $x\neq 0$ in $\mathfrak{F}$ that $m(x)=t^{n(x)}(t-f(x))$, where $f:\mathfrak{F}\to \R$ is a non-trivial linear functional and $n:\mathfrak{F}\bs\{0\}\to \{2,3\}$ with the property that 
\begin{itemize}
 \item $\ker f\subset n^{-1}(2)$ and 
 \item $n^{-1}(3)\neq \emptyset$ .
\end{itemize}
Then $\mathfrak{F}$ is conjugate to either $\mathfrak{L'}$ or $\mathfrak{L'_{-}}$.
\end{lemma}

\begin{proof}
 We begin by selecting generators $\alpha,\beta$ for $\mathfrak{F}$ such that $\alpha\in n^{-1}(3)$ (observe that our hypothesis forces $f(\alpha)\neq0$). Using (a small variation of) Jordan normal form, we can select a basis where 
 
  $$\alpha=\begin{pmatrix}
    0 & 0 & 1 & 0\\
    0 & f(\alpha) & 0 & 0\\
    0 & 0 & 0 & 1\\
    0 & 0 & 0 & 0
   \end{pmatrix}\text{ and }
\beta=\begin{pmatrix}
       b_{11} & b_{12} & b_{13} & b_{14}\\
       b_{21} & b_{22} & b_{23} & b_{24}\\
       b_{31} & b_{32} & b_{33} & b_{34}\\
       b_{41} & b_{42} & b_{43} & b_{44}
      \end{pmatrix}
$$
 
 Computing the commutator we see that 
$$[\alpha,\beta]=\begin{pmatrix}
b_{31} & b_{32}-f(\alpha)b_{12} & b_{33}-b_{11} & b_{34}-b_{13}\\
f(\alpha)b_{21} & 0 & f(\alpha)b_{23}-b_{21} & f(\alpha)b_{24}-b_{23}\\
b_{41} & b_{42}-f(\alpha)b_{32} & b_{43}-b_{31} & b_{44}-b_{33}\\
0 & -f(\alpha)b_{42} & -b_{41} & -b_{43}
\end{pmatrix}.$$
Therefore, $b_{12}=b_{21}=b_{23}=b_{24}=b_{31}=b_{32}=b_{41}=b_{42}=b_{43}=0$,  $b_{11}=b_{33}=b_{44}$, and $b_{13}=b_{34}$. We conclude that 
$$\beta=\begin{pmatrix}
b_{11} & 0 & b_{13} & b_{14}\\
0 & b_{22} & 0 & 0 \\
0 & 0 & b_{11} & b_{13}\\
0 & 0 & 0 & b_{11}
\end{pmatrix}.$$
 From the properties of $m$ we deduce that $b_{11}=0$ and $b_{22}=f(\beta)$, and thus

  we have now reduced to the case that $\mathfrak{F}$ is conjugate to a Lie algebra of the form 
  $$\begin{pmatrix}
     0 & 0 & b &c_1a+c_2b\\
     0 & a & 0 & 0\\
     0 & 0 & 0 & b\\
     0 & 0 & 0 & 0
    \end{pmatrix}
$$
$c_1\neq0$ since otherwise there would be an element of $\mathfrak{F}$ whose minimal polynomial is not divisible by $t^2$. Finally, by conjugating by 
$$\begin{pmatrix}
   \abs{c_1} & 0 & 0 & 0\\
   0 & 1& 0 & 0\\
   0 & 0 & \sqrt{\abs{c_1}} & -c_2\\
   0 & 0 & 0 & 1
  \end{pmatrix}
$$
we can assume that $c_2=0$ and $c_1=\pm 1$. 
\qed
\end{proof}

Next, we will address the question of when a family, $\Gamma_t$, of lattices in $L'$ can be conjugated so that their limit is a lattice of $\Gamma_0\subset L_0$. Since $L'$ is a simply connected abelian Lie group we see that it is isomorphic to its Lie algebra (via the exponential map), which is in turn isomorphic to $\R^2$. Explicitly, we have a map 
\begin{equation}\label{e:liealgebramap}
(a,b)\stackrel{g}{\mapsto} \begin{pmatrix}
                0 & 0 & b & -a\\
                0 & a & 0 & 0\\
                0 & 0 & 0 & b\\
                0 & 0 & 0 & 0
               \end{pmatrix}\stackrel{\exp}{\mapsto}
\begin{pmatrix}
                1 & 0 & b & \frac{b^2}{2}-a\\
                0 & e^a & 0 & 0\\
                0 & 0 & 1 & b\\
                0 & 0 & 0 & 1
               \end{pmatrix}
\end{equation}

We can now give a sufficient condition for a family of lattices in $L'$ to converge after conjugation to a lattice in $L_0$.
\begin{proposition}\label{p:Ltconvergence}
$\Gamma_t=\langle A_t,B_t\rangle$ be a family of lattices in $L'$, let $a_t$ and $b_t$ be the vectors in $\R^2$ that correspond to $A_t$ and $B_t$ under \eqref{e:liealgebramap}. Suppose that
\begin{enumerate}
 \item $\lim_{t\to 0}a_t=\lim_{t\to 0}b_t=0$,
 \item $a_t$ and $b_t$ are differentiable near 0, and 
 \item the derivatives of $a_t$ and $b_t$ at $0$ are linearly independent vectors.
\end{enumerate}
Then there exists a family $C_t$ of invertible matrices such that $C_t \Gamma_t C_t^{-1}\subset L_t$ and $\lim_{t\to 0} C_t \Gamma_t C_t^{-1}$ is a lattice in $L_0$. 
\end{proposition}
\begin{proof}
Let $C_t=V_t$  be the family of matrices from \eqref{e:Vt}. As previously mentioned, the matrix $V_t$ conjugates the Lie algebra $\mathfrak{L}'$ to the Lie algebra $\mathfrak{L}_t$. Explicitly we see that if $x\in \mathfrak{L}'$ corresponds to the vector $(x_1,x_2)\in \R^2$ via $g$ then
\begin{equation}\label{e:conjfamily}
C_t x C_t^{-1}=\begin{pmatrix}
                 0 & \frac{x_1}{t} & \frac{x_2}{t} & 0\\
                 0 & x_1 & 0 & \frac{x_1}{t}\\
                 0 & 0 & 0 & \frac{x_2}{t}\\
                 0 & 0 & 0 & 0
                 \end{pmatrix}
\end{equation}
From \eqref{e:conjfamily} we see that $C_t\Gamma_tC_t^{-1}\subset L_t$. Furthermore, $C_t\Gamma_tC_t^{-1}\subset L_t$ will converge as $t\to 0$ provided that the limits of the entries of $C_tA_tC_t^{-1}$ and $C_tB_tC_t^{-1}$ converge as $t\to 0$. Since $\exp$ is a smooth map  we see from \eqref{e:conjfamily} that $C_tB_tC_t^{-1}$ and $C_tA_tC_t^{-1}$ will converge provided that conditions (1) and (2) are satisfied. Furthermore, if condition (3) is satisfied then it is easy to see that the limits of $C_tA_tC_t^{-1}$ and $C_tB_tC_t^{-1}$ will generate a rank two group.  
\qed 
\end{proof}

\section{Obstructions Coming From Cusp Shape}\label{s:obstruction}

In this section we discuss the relationship between the cusp shape of a finite volume hyperbolic manifold and types of projective deformations that can occur. 

For the sake of completeness we will briefly describe the structure of finite volume hyperbolic 3-manifolds and cusp shapes. For details about the structure of hyperbolic manifolds see \cite{Ratcliffe06, ThurstonNotes} and for details about cusp shape see \cite{Riley79}. Let $M$ be an orientable 3-manifold that admits a complete, finite volume, hyperbolic structure then $M$ can be written as 
\begin{equation}\label{e:hyperbolicdecomp}
M=M_K\cup\bigcup_{j=1}^m C_j,
\end{equation}
where $M_K$ is a compact manifold with $m$ torus boundary components, each $C_j$ is diffeomorphic to $T^2\times [1,\infty)$, where $T^2$ is a 2-dimensional torus, and $M_K\cap C_j=T^2\times\{1\}$. Furthermore, each $C_j$ is a convex submanifold of $M$ with strictly convex boundary. The interior of each $C_j$ is projectively equivalent to $D_0/\Gamma_0$, where $\Gamma_0$ is a  lattice in $L_0$. Therefore, each $C_j$ is a generalized cusp. We refer to the $C_j$ as \emph{cusps} and after picking a basepoint in $M$ we can identify $\fund{C_j}\cong \Z^2$ with a subgroup of $\fund{M}$ and we refer to these subgroups as \emph{peripheral subgroups}. Next, let $\Delta_j$ be a peripheral subgroup associated with the cusp $C_j$, then the finite volume hyperbolic structure provides a well defined Euclidean similarity structure associated to $C_j$ called the \emph{cusp shape} of $C_j$. 

For each $j$ we choose generators $m_j$ and $l_j$ for $\Delta_j$. The cusp shape induces a discrete faithful representation from $\Delta_j$ into $\psl$. After conjugating, we can assume that under this representation 
$$m_j\mapsto \begin{pmatrix}
            1 & 1\\
            0 & 1
           \end{pmatrix}\ 
l_j\mapsto \begin{pmatrix}
          1 & \omega_j\\
          0 & 1
         \end{pmatrix}
$$
We call the number $\omega_j$ the \emph{cusp shape of $C_j$ relative to $\{m_j,l_j\}$} (see \cite{Riley79} for a proof that this complex number is well defined).  Furthermore, by choosing $m_j$ and $l_j$ to be properly oriented with respect to orientation on $C_j$ coming from $M$ we can assume that $\omega_j$ has positive imaginary part. If $m'_j$ and $l'_j$ are another set of generators that are positively oriented with respect to the induced orientation coming form $M$ then the cusp shape of $C_j$ with respect to $\{m'_j,l'_j\}$ is in the same $PSL_2(\Z)$-orbit as $\omega_j$ (here we are assuming that $PSL_2(\Z)$ acts on $\C$ by linear fractional transformations). With this in mind we say that $C_j$ \emph{has imaginary cusp shape } if for some choice of generators of $\Delta_j$ the cusp shape with respect to these generators has real part equal to 0.   

Let $P$ be the subgroup of $\psl$ of matrices of the form 
$$\begin{pmatrix}
   1 & x+iy\\
   0 & 1
  \end{pmatrix}
$$
There is an isomorphism between $L_0$ and $P$ given by 
\begin{equation}\label{e:L_0toP}
 \begin{pmatrix}
  1 & x & y & \frac{1}{2}\left(x^2+y^2\right)\\
  0 & 1 & 0 & x\\
  0 & 0 & 1 & y\\
  0 & 0 & 0 & 1
 \end{pmatrix}\mapsto
\begin{pmatrix}
 1 & x+iy\\
 0 & 1
\end{pmatrix}
\end{equation}
If we have a complete hyperbolic structure on $M$ then \eqref{e:L_0toP} tells us that there is an induced representation of $\Delta_j$ into $L_0$, where 
$$m_j\mapsto \begin{pmatrix}
            1 & x^j_1 & y^j_1 & \frac{1}{2}\left((x^j_1)^2+(y^j_1)^2\right)\\
            0 & 1 & 0 & x^j_1\\
            0 & 0 & 1 & y^j_1\\
            0 & 0 & 0 & 1
           \end{pmatrix} \ 
l_i\mapsto \begin{pmatrix}
            1 & x^j_2 & y^j_2 & \frac{1}{2}\left((x^j_2)^2+(y^j_2)^2\right)\\
            0 & 1 & 0 & x^j_2\\
            0 & 0 & 1 & y^j_2\\
            0 & 0 & 0 & 1
           \end{pmatrix}$$
Thus we see that the cusp shape of the $j$th cusp of $M$ relative to $\{m_j,l_j\}$ is given by 
\begin{equation}\label{e:cuspshape}
\frac{x_2^j+iy_2^j}{x_1^j+iy_1^j}=\frac{1}{(x^j_1)^2+(y^j_1)^2}\left(x^j_1x^j_2+y^j_1y^j_2+i\left(x^j_1y^j_2-x^j_2y^j_1\right)\right). 
\end{equation}

The next proposition uses \eqref{e:cuspshape} to show how the cusp shape of a manifold can obstruct certain types of deformations. 
\begin{proposition}\label{p:cuspshapeobstruction}
 Let $M$ be a non-compact finite volume hyperbolic 3-manifold. Let $C$ be a cusp of $M$ and let  $\{m,l\}$ be a choice of generators for a peripheral subgroup corresponding to $C$.  Suppose that $M$ admits a family of projective structures with holonomy $\rho_t$ such that $\rho_0=\rhogeo$ and $\rho_t(m)$ is a pure translation in $L_t$ and $\rho_t(l)$ is a pure dilation in $L_t$. Then the cusp shape of $C$ is purely imaginary. 
\end{proposition}
\begin{proof}
 Since $\rho_t(m)$ is a pure translation and $\rho_t(l)$ is a pure dilation, we see that there exists functions $\mu_t$ and $\nu_t$, continuous near zero, such that 
 $$\rho_t(m)=\exp\left(\begin{pmatrix}
              0 & 0 & \mu_t & 0\\
              0 & 0 & 0 & 0\\
              0 & 0 & 0 & \mu_t\\
              0 & 0 & 0 & 0
             \end{pmatrix}\right) \textrm{ and }
\rho_t(l)=\exp\left(\begin{pmatrix}
                     0 & \nu_t & 0 & 0\\
                     0 & \nu_t t & 0 & \nu_t\\
                     0 & 0 & 0 & 0\\
                     0 & 0 & 0 & 0
                    \end{pmatrix}
\right).$$
Thus if we let $\mu_0=\lim_{t\to 0}\mu_t$ and $\nu_0=\lim_{t\to0}\nu_t$, then we see that 
$$\rho_0(m)=\exp\left(\begin{pmatrix}
                       0 & 0 & \mu_0 & 0\\
                       0 & 0 & 0 & 0\\
                       0 & 0 & 0 & \mu_0\\
                       0 & 0 & 0 & 0
                      \end{pmatrix}
\right)=\begin{pmatrix}
         1 & 0 & \mu_0 & \frac{1}{2}\mu_0^2\\
         0 & 1 & 0 & 0\\
         0 & 0 & 1 & \mu_0\\
         0 & 0 & 0 & 1
        \end{pmatrix} \textrm{ and }$$         
$$\rho_0(l)=\exp\left(\begin{pmatrix}
                     0 & \nu_0 & 0 & 0\\
                     0 & 0 & 0 & \nu_0\\
                     0 & 0 & 0 & 0\\
                     0 & 0 & 0 & 0
                    \end{pmatrix}
\right)=\begin{pmatrix}
         1 & \nu_0 & 0 & \frac{1}{2}\nu_0^2\\
         0 & 1 & 0 & \nu_0\\
         0 & 0 & 1 & 0\\
         0 & 0 & 0 & 1
        \end{pmatrix}
.$$
From \eqref{e:cuspshape} we see that the cusp shape of $C$ relative to $\{m,l\}$ is $-i\frac{\nu_0}{\mu_0}$, and is thus imaginary. 
\qed
\end{proof}

In \cite{Ballas12a} it is shown that the complements of the knots $5_2$ and $6_1$ in $S^3$ do not admit a families $\rho_t$ of representations passing through $\rhogeo$  such that the image under $\rho_t$ of the meridian is a pure translation and the image under $\rho_t$ of the longitude is a pure dilation. On the other hand, we have seen that the figure-eight knot complement does admit a family $\rho_t$ of representations passing through $\rhogeo$ such that the image under $\rho_t$ of the meridian is a pure translation and the image under $\rho_t$ of the longitude is a pure dilation. The cusp shape of $5_2$ and $6_1$ are not purely imaginary, but the cusp shape of the figure-eight knot is purely imaginary. Proposition \ref{p:cuspshapeobstruction} can be seen as a partial explanation of this phenomenon.  

\section{Volume of Cusps}\label{s:volume}

In this section we will examine the Busemann volume of cusps that are projectively equivalent to $C_\Gamma=D'/\Gamma$, where $\Gamma$ is lattice inside $L'$. While $C_\Gamma$ will always have infinite volume, we will show that $C_\Gamma$ admits an exhaustion by sets of finite Busemann volume. The domain $D'$ can be written as 
$$\bigcup_{s>0}\mathcal{H}'_s,$$
(see Figure \ref{f:domain1}). For each $k>0$ we let $D'_k=\cup_{s>k}\mathcal{H}'_s$, and it is easy to see that  $\{D'_k\}_{k>0}$ is an exhaustion of $D'$ by horoballs. Each $D_k'$ is preserved by $L'$ and the main result of this section is that some (hence any) fundamental domain for the action of $\Gamma$ on $D'_k$ has finite Busemann volume when regarded as a subset of $D'$.

As mentioned in Section \ref{s:Z^2}, $\mathcal{H}'_k/\Gamma$ is an affine torus whose developing map is given by the restriction of $[x_1:x_2:x_3:1]\mapsto [x_2:x_3:1]$ to $\mathcal{H}'_k$. Let $R$ be a compact fundamental domain for the affine action of $\Gamma$ on affine 2-space. Then a fundamental domain for $D'_k/\Gamma$ can be taken to be 
$$\mathcal{D}_k=\{[x_1:x_2:x_3:1]\mid (x_2,x_3)\in R, x_1>0\}\cap D'_k.$$
Geometrically, $\mathcal{D}_k$ is the intersection of $D'_k$ with the cone over $R$ whose cone point is ${[1:0:0:0]}$.

\begin{figure}
 \begin{center}
  \includegraphics[scale=.6]{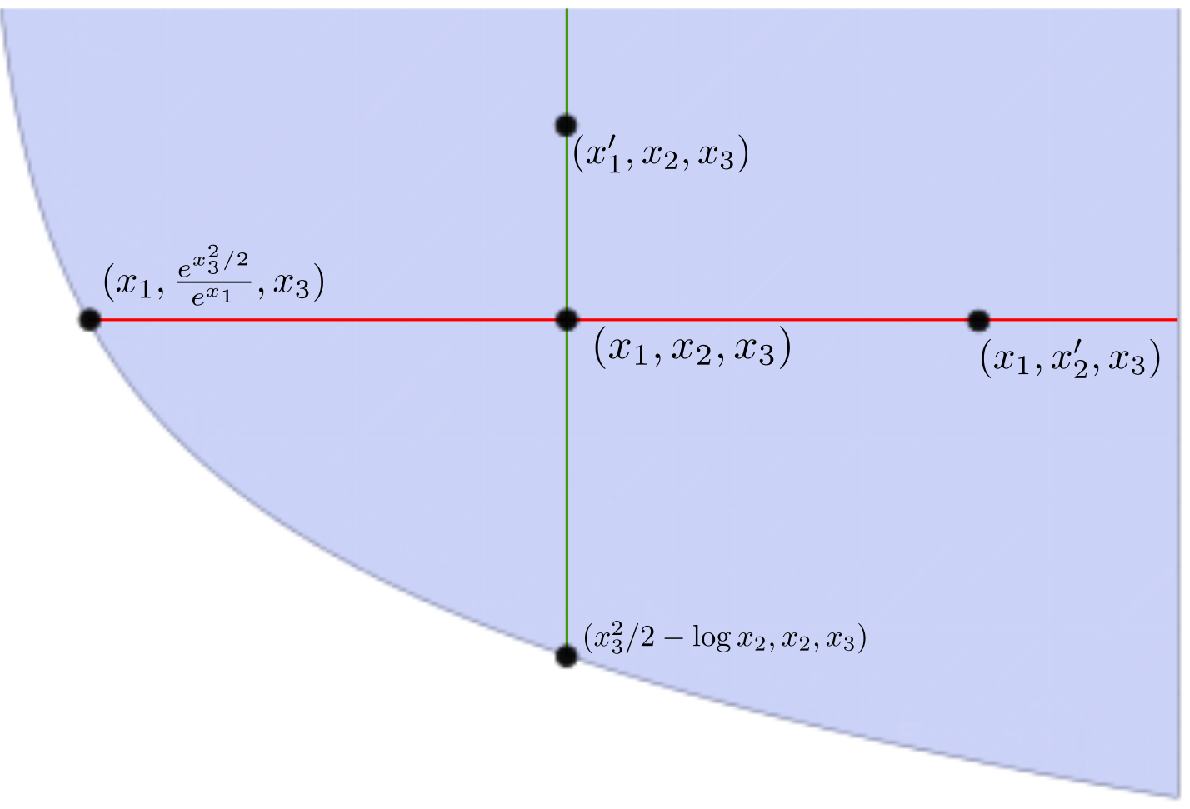}\hspace{.5in}
  \includegraphics[scale=.8]{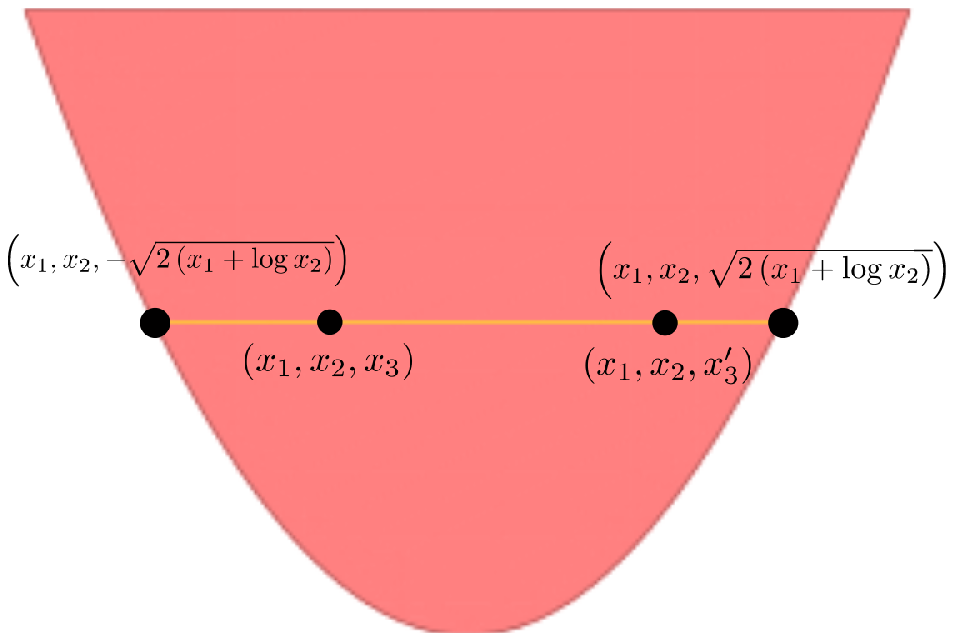}
  \caption{Slices of $D'$ in the $x_2$ and $x_3$ directions\label{f:slices}}
 \end{center}
 \end{figure}

\begin{lemma}\label{l:euclideanvolume}
 Let $\Gamma$ be a lattice in $L'$, let $k>0$, and let $x=[x_1:x_2:x_3:1]$ be a point in $\mathcal{D}_k$. There exist constants $C$ and $N$ (depending only on $\Gamma$) such that if $x_1>N$ then $Cx_1^{3/2}<\mu_L(B^{D'}_x(1)).$ 
\end{lemma}

\begin{proof}
 We write $x=(x_1,x_2,x_3)$ as a vector in $\R^3$.  The basic idea of the proof is to show that when $x_1$ is large that $B_x^{D'}(1)$ contains a simplex of large Lebesgue volume. Let $v^2_c=x+ce_2$. From Figure \ref{f:slices} and \eqref{e:hilbertnorm} we see that
 \begin{equation}\label{e:distform1}
  \Abs{v^2_c}_{D'}=\left.\frac{d}{dt}\right|_{t=0}\log\left(\frac{x_2+ct-k_1}{x_2-k_1}\right)=\frac{c}{x_2-k_1},
 \end{equation}
where $k_1=k_1(x_1,x_3):=\frac{e^{x_3^2/2}}{e^{x_1}}$. Since $(x_2,x_3)\in R$ and $R$ is compact we know that there are positive constants $c_1,c_2$, and $c_3$ such that $c_1<x_2<c_2$ and $-c_3<x_3<c_3$. Combining this fact with \eqref{e:distform1} we see that we can find a small positive constant $T$ such that when $x_1$ is sufficiently large $\Abs{v^2_T}_{D'}<1$. As a consequence $v^2_T\in B_x^{D'}(1)$. 

Let $v^1_c=x+ce_1$. From Figure \ref{f:slices} and \eqref{e:hilbertnorm} we see that 
\begin{equation}\label{e:distform2}
 \Abs{v^1_c}_{D'}=\left.\frac{d}{dt}\right|_{t=0}\log\left(\frac{x_1+ct-k_2}{x_1-k_2}\right)=\frac{c}{x_1-k_2},
\end{equation}
where $k_2=k_2(x_2,x_3):=\frac{1}{2}x_3^2-\log x_2$. When $x_1$ is sufficiently large we see that $\frac{x_1}{2(x_1-k_2)}<1$ and so $v^1_{x_1/2}\in B^{D'}_x(1)$. 

Finally, let $v^3_c=x+ce_3$. From Figure \ref{f:slices} and \eqref{e:hilbertnorm} we see that
\begin{equation}\label{e:distform3}
 \Abs{v^3_c}_{D'}=\left.\frac{d}{dt}\right|_{t=0}\log\left(\frac{(x_3+ct+k_3)(k_3-x_3)}{(x_3+k_3)(k_3-x_3-ct)}\right)=c\left(\frac{1}{(x_3+k_3)}+\frac{1}{(k_3-x_3)}\right),
\end{equation}
 where $k_3=k_3(x_1,x_2):=\sqrt{2\left(x_1+\log x_2\right)}$. When $x_1$ is sufficiently large we see that $\frac{\sqrt{x_1}}{3\sqrt{2}}\left(\frac{1}{(x_3+k_3)}+\frac{1}{(k_3-x_3)}\right)=\frac{\sqrt{x_1}}{3\sqrt{2}}\left(\frac{2k_3}{k_3^2-x_3^2}\right)<1$ and so $v^3_{\sqrt{x_1}/3\sqrt{2}}\in B^{D'}_x(1)$.
 
 Thus for sufficiently large $x_1$, we see that $B^{D'}_x(1)$ contains $v^2_{T}$, $v^1_{x_1/2}$ and $v^3_{\sqrt{x_1}/3\sqrt{2}}$. Since $B^{D'}_x(1)$ is the unit ball for a norm we see that it is convex and thus it contains the convex hull of the set $\left\{x,v^2_T,v^1_{x_1/2},v^3_{\sqrt{x_1}/3\sqrt{2}}\right\}$. This convex hull is a tetrahedron and its Lebesgue measure is easily computed to be $\frac{x_1^{3/2}}{36\sqrt{2}}T$ and we conclude that when $x_1$ is sufficiently large $\frac{x_1^{3/2}}{36\sqrt{2}}T<\mu_L(B_x^{D'}(1))$
 
 \qed 
\end{proof}

We can now prove the main result of this section.

\begin{proposition}\label{t:finitevolume}
 $\mathcal{D}_k$ has finite Busemann volume when regarded as a subset of $D'$. 
\end{proposition}
\begin{proof}
 From \eqref{e:hilbertvolume} we see that $\hilbvol[D']{\mathcal{D}_k}=\int_{\mathcal{D}_k}\frac{d\mu_L(x)}{\mu_L(B^{D'}_x(1))}$. By Lemma \ref{l:euclideanvolume} we know that there exists a compact set $K\subset \mathcal{D}_k$ and a constant $C$ such that $Cx_1^{3/2}<\mu_L(B^{D'}_x(1))$ for every $x=(x_1,x_2,x_3)\in \mathcal{D}_k\bs K$. Thus we see that 
 $$\hilbvol[D']{\mathcal{D}_k}=\int_K\frac{d\mu_L(x)}{\mu_L(B^{D'}_x(1))}+\int_{\mathcal{D}_k\bs K}\frac{d\mu_L(x)}{\mu_L(B^{D'}_x(1))}<\int_K\frac{d\mu_L(x)}{\mu_L(B^{D'}_x(1))}+\int_{\mathcal{D}_k\bs K}\frac{d\mu_L(x)}{Cx_1^{3/2}}<\infty$$
\qed 
\end{proof}

\section{Figure-eight example}\label{s:fig8}

In this section we use the ideas from the previous section along with work of \cite{CooperLongTillman13} to exhibit an explicit path of pairwise inequivalent finite volume properly convex projective structures that passes through the complete hyperbolic structure of the figure-eight knot complement. 

Let $M$ be the complement of the figure-eight knot. Next, choose a basepoint $y\in M$, and let $\Gamma=\fund{M}$ with respect to this choice of basepoint. The group $\Gamma$ is generated by two elements $m$ and $n$, which are freely homotopic to meridians of the knot. Explicitly, the group $\Gamma$ can be written as 
$$\Gamma=\langle m,n\vert mw=wn\rangle,$$
where $w=nm^{-1}n^{-1}m$. As mentioned in Section \ref{s:obstruction}, $M$ can be written as $M_K\cup C$ where $M_K$ is compact and $C$ is diffeomorphic to $T^2\times [0,\infty)$ and we let $\Delta=\fund{C}$ be a choice of peripheral subgroup. 

The complete hyperbolic structure of $M$ induces a representation $\rhogeo:\Gamma\to \text{P}\so{3}\subset \pgl$. Under this representation, $\rhogeo(\Delta)$ is conjugate to a lattice inside $L_0$ as described in Section \ref{s:Z^2}.  For the figure-eight knot complement, the group $\Delta$ can be generated by $m$ and the element $l=ww^{op}=nm^{-1}n^{-1}m^2n^{-1}m^{-1}n$, where $w^{op}$ is the word $w$ written backwards. It is easy to see that $l$ is homologically trivial and that $l$ corresponds to a longitude of the knot.

In \cite{Ballas12a} an explicit family, $\rho_t$, of representations from $\Gamma$ into $\pgl$ is found for which $\mathcal{M}_t=\rho_t(m)$ and $\mathcal{N}_t=\rho_t(n)$ are both unipotent matrices and $\rho_{1/2}$ is the holonomy of the complete hyperbolic structure on $M$. Additionally, $\mathcal{L}_t=\rho_t(l)$ is unipotent if and only if $t=\frac{1}{2}$. We now show that the restriction of $\rho_t$ to the peripheral subgroup is the holonomy of a properly convex projective structure on $C$ which converges to the hyperbolic structure on $C$ coming from $\rho_{1/2}$. Specifically, $\rho_t$ is given by  

$$\mathcal{M}_t=\begin{pmatrix}
       1 & 0 & 1 & t-1\\
       0 & 1 & 1 & t\\
       0 & 0 & 1 & t+\frac{1}{2}\\
       0 & 0 & 0 & 1
      \end{pmatrix} \textrm{ and }
  \mathcal{N}_t=\begin{pmatrix}
       1 & 0 & 0 & 0\\
       2+\frac{1}{t} & 1 & 0 & 0\\
       2 & 1 & 1 & 0\\
       1 & 1 & 0 & 1
      \end{pmatrix}$$
$$\mathcal{L}_t=\begin{pmatrix}
       \frac{8t^3-4t^2-2t-1}{8t^2} & \frac{8t^3+4t^2+2x+1}{8t^2} & \frac{-4t^2-1}{4t^2} & \frac{40t^3+24t^2+4t+3}{8t^2}\\
       \frac{8t^4-4t^3-2t^2-t-1}{8t^3} & \frac{8t^4+4t^3+2t^2+t+1}{8t^3} & \frac{4t^3-4t^2+t-1}{4t^3} & \frac{56t^4+16t^3+20t^2+t+3}{8t^3}\\
       0 & 0 & 2t & 0 \\
       0 & 0 & 0 & 2t
      \end{pmatrix}.
$$

Let $\Delta_t=\rho_t(\Delta)=\langle \mathcal{M}_t,\mathcal{L}_t\rangle$. We now replace $\mathcal{L}_t$ with the projectively equivalent matrix $\frac{1}{2t}\mathcal{L}_t$. After performing the conjugacy described in the proof of Lemma \ref{l:liealgconjugacy} and applying the coordinate change $s=\log\left(\frac{1}{16t^4}\right)$ we see that $\mathcal{M}_s$ and $\mathcal{L}_s$ are conjugate to 
$$\begin{pmatrix}
       1 & 0 & \sqrt{\frac{s\sinh(s/4)}{3}} & \frac{s\sinh(s/4)}{6}\\
       0 & 1 & 0 & 0\\
       0 & 0 & 1 & \sqrt{\frac{s\sinh(s/4)}{3}}\\
       0 & 0 & 0 & 1
      \end{pmatrix}\textrm {and}
      \begin{pmatrix}
       1 & 0 & 0 & -s\\
       0 & e^s & 0 & 0\\
       0 & 0 & 1 & 0\\
       0 & 0 & 0 & 1
      \end{pmatrix}$$
respectively. We can further conjugate $\mathcal{M}_s$ and $\mathcal{L}_s$ to the matrices

\begin{equation}\label{e:fig8cuspconvergence}
\mathcal{M}'_s=\begin{pmatrix}
       1 & 0 & \sqrt{\frac{\sinh(s/4)}{3s}} & \frac{\sinh(s/4)}{6s}\\
       0 & 1 & 0 & 0\\
       0 & 0 & 1 & \sqrt{\frac{\sinh(s/4)}{3s}}\\
       0 & 0 & 0 & 1
      \end{pmatrix} \textrm{ and }
\mathcal{L}'_s=\begin{pmatrix}
       1 & \frac{e^s-1}{s} & 0 & \frac{e^s-s-1}{s^2}\\
       0 & e^s & 0 & \frac{e^s-1}{s}\\
       0 & 0 & 1 & 0\\
       0 & 0 & 0 & 1
      \end{pmatrix}.
\end{equation}
A Mathematica notebook containing these calculations can be found at \cite{BallasNotebook}.

From the discussion in Section \ref{s:Z^2} we see that for each $s\neq 0$ that $\Delta_s$ is conjugate to the group $\Delta_s'=\langle \mathcal{M}'_s,\mathcal{L}'_s \rangle$ that preserves $D_s$ and such that $D_s/\Delta'_s\cong C$. Thus we see that the restriction of $\rho_s$ to $\Delta$ gives $C$ the structure of a generalized cusp.

As $s\to 0$ ($t\to\frac{1}{2}$) $\mathcal{M}'_s$ and $\mathcal{L}'_s$ converge to the matrices
\begin{equation}\label{e:limitmatrix}
\mathcal{M}_0=\begin{pmatrix}
           1 & 0 & \frac{1}{2\sqrt{3}} & \frac{1}{24}\\
           0 & 1 & 0 & 0\\
           0 & 0 & 1 & \frac{1}{2\sqrt{3}}\\
           0 & 0 & 0 & 1
          \end{pmatrix}\textrm{ and }
          \mathcal{L}_0=\begin{pmatrix}
           1 & 1 & 0 & \frac{1}{2}\\
	   0 & 1 & 0 & 1\\
	   0 & 0 & 1 & 0\\
	   0 & 0 & 0 & 1
          \end{pmatrix}. 
\end{equation}
In other words, as $s\to0$ the group $\Delta'_s$ is converging to a lattice $\Delta_0=\langle \mathcal{M}_0,\mathcal{L}_0 \rangle \leq L_0$  and $D_s/\Delta'_s$ is converging to a hyperbolic cusp $D_0/\Delta_0$ (see Figure \ref{f:domaindeformation}).  
By looking at the entries of the matrices in \eqref{e:limitmatrix} and applying the formula \eqref{e:cuspshape} we see the cusp shape of the limit structure is $-2\sqrt{3} i$, which is the cusp shape of of the figure-eight knot. Thus the projective structure on $C$ coming from $\rho_s$ is converging the hyperbolic structure on $C$ coming from $\rho_{0}$. 

The next result shows that for small values of $s$, $\rho_s$ is the holonomy of a properly convex projective structure.

\begin{proposition}\label{p:convexstructure}
 There exists $\varep>0$ such that for $s\in (-\varep,\varep)$, $\rho_s$ is the holonomy of a properly convex structure on $M$. Furthermore, this structure is strictly convex if and only if $s=0$.
\end{proposition}

\begin{proof}
When $s=0$ then $\rho_0$ is the holonomy of the complete hyperbolic structure on $M$. As a result the $\rho_0$ is the holonomy of a strictly convex projective structure on $M$. From \eqref{e:hyperbolicdecomp} we see that $M=A\cup B$ where $A$ compact and $B$ has the structure of a generalized cusp. When $s\neq 0$ we have seen that $\Delta_s=\rho_s(\fund{B})$ is conjugate to a lattice in $L'$. If $\mathcal{B}$ is an $L'$-horoball then $\Delta_s$ is the holonomy of the generalized cusp $\mathcal{B}/\Delta_s$. By applying Theorem \ref{t:holonomytheorem} we see that for sufficiently small $s$, $\rho_s$ will the the holonomy of a properly convex projective structure on $M$. 

When $s\neq 0$, a simple calculation shows that  the matrix $\mathcal{L}_s$ has two distinct positive real eigenvalues and is thus not projectively equivalent to a parabolic transformation. By \cite[Prop 3.2.4]{BallasThesis} we see that this implies that $\rho_s$ cannot be the holonomy of a strictly convex projective structure on $M$. 

\qed 
\end{proof}

From Proposition \ref{p:convexstructure} we know that for small values of $s$, $\rho_s$ is the holonomy of a properly convex projective structure on $M$. To complete the proof of Theorem \ref{t:maintheorem} it remains to prove that this structure has finite Busemann volume. 

Let $\Omega_s$ be the properly convex domain preserved by $\Gamma_s:=\rho_s(\Gamma)$. Assume that we have conjugated so that $\Delta_s\subset L'$ and choose parabolic coordinates for $\Omega_s$ centered at $(p,H)$, where $p=[1:0:0:0]$ and the $H$ is the hyperplane dual to $[0:0:0:1]$. In these affine coordinates we have the $x_1$ direction is the vertical direction. We can further assume that the affine action of the pure translations on $H_0$ is by translation in the $x_3$ direction and the affine action of the pure dilations on $H_0$ is by dilation in the $x_2$ direction. The basic idea is that even though we do not know exactly what $\Omega_s$ looks like we can approximate $\Omega_s$ using $L'$-horoballs. We then use our geometric understanding of $D'$ to deduce information about $\Omega_s$.

\begin{lemma}\label{l:horoapprox}
There exist $L'$-horoballs $\mathcal{B}$ and $\mathcal{B}'$ so that in parabolic coordinates for $\Omega_s$ centered at $(p,H)$ we have $\mathcal{B}\subset \Omega_s\subset \mathcal{B}'$.  
\end{lemma}

\begin{proof}
Let $f$ be the boundary function $\Omega_s$ and $g$ be the boundary function for $D'$. The domain, $U$, of $g$ is one of the two half spaces in $H_0$ that form the complement of the plane defined by the equation $x_2=0$. We claim that the domain of $f$ is $U$. The only $\Delta_s$ invariant convex subsets of $H_0$ with non-empty interior are $U$, $\overline{U}$, and $H_0$. Suppose the domain of $f$ is not $U$, then there is a point of $\partial U$ in the domain of $f$. This implies that there is a point of the form $[c:0:d:1]\in \Omega_s$. Since the action of a pure translation on this point is by translation in the vertical direction it is easy to see that the convex hull of the $\Delta_s$ orbit of this point is the affine plane whose vertical projection is $\partial U$. This contradicts the fact that $\Omega_s$ is properly convex and so the domain of $f$ is also $U$.   

Since every $L'$-horoball is simply the epigraph of the function $g+c$ where $c\in \R$ the proof will be complete if we can find positive constants $c$ and $C$ such that for all $x\in U$ $g(x)-c< f(x)< g(x)+C.$  We have seen that there is cocompact affine action of $\Delta_s$ on $U$ whose quotient is a torus. Let $K\subset U$ be a compact fundamental domain for this action. Pick $c$ and $C$ to be constants such that $(g-c)|_K< f|_K< (g+C)|_K$. 

Suppose for contradiction that there exists a point $x\in U$ such that $g(x)-c\geq f(x)$. Then by the intermediate value theorem we can find a point $x'\in U$ such that $g(x')-c=f(x)$. There exists $\gamma\in \Delta_s$ such that $\gamma\cdot x'\in K$. We see that $(f(x'),x')=(g(x')-c,x')$ and so $\gamma\cdot (f(x'),x')=\gamma\cdot(g(x')-c,x').$ By equivariance we see that $(f(\gamma\cdot x'),\gamma\cdot x')=(g(\gamma\cdot x')-c,\gamma\cdot x')$. This implies that $f(\gamma\cdot x')=g(\gamma\cdot x')-c$, which contradicts our choice of $c$ and thus $g-c<f$. A similar argument shows that $f<g+C$.  
\qed 
\end{proof}

Let $G$ be a group acting on a set $X$. If $Y\subset X$ and $H$ is a subgroup then we say that $Y$ is \emph{precisely invariant under $H$} if $\gamma\cdot Y=Y$ for all $\gamma \in H$ and $\gamma\cdot Y\cap Y=\emptyset$ for $\gamma\in G\bs H$. By combining the previous lemma with a version of the Margulis lemma for properly convex domains \cite[Thm 0.1]{CooperLongTillman11} we can show that there are precisely invariant $L'$ horoballs in $\Omega_s$.

\begin{lemma}\label{l:precinvhoroball}
For sufficiently small $s$ there is an $L'$-horoball, $\mathcal{B}\subset \Omega_s$ that is precisely invariant under $\Delta_s$. 
\end{lemma}

\begin{proof}
Let $\mu$ be a 3-dimensional properly convex Margulis constant, which exists by \cite[Thm 0.1]{CooperLongTillman11}. By Lemma \ref{l:horoapprox} we know that we can find an $L'$-horoball $\mathcal{B}'\subset \Omega_s$ and let $\mathcal{B}''$ be a $L'$-horoball contained in the interior of $\mathcal{B}'$. Let $z\in \partial \mathcal{B}''$. We claim that for every $y\in \partial \mathcal{B}''$ $\dhilb[\mathcal{B}']{y}{\mathcal{M}_s y}=\dhilb[\mathcal{B}']{z}{\mathcal{M}_s z}$. To see this observe that $\Delta_s$ acts transitively on $\partial \mathcal{B}''$ and so there exists $\gamma\in \Delta_s$ such that $\gamma z=y$. From this fact we deduce that 
$$\dhilb[\mathcal{B}']{y}{\mathcal{M}_s y}=\dhilb[\mathcal{B}']{\gamma z}{\mathcal{M}_s\gamma z}=\dhilb[\mathcal{B}']{\gamma z}{\gamma \mathcal{M}_s z}=\dhilb[\mathcal{B}']{z}{\mathcal{M}_s z}.$$

\begin{figure}
 \begin{center}
  \includegraphics[scale=.2]{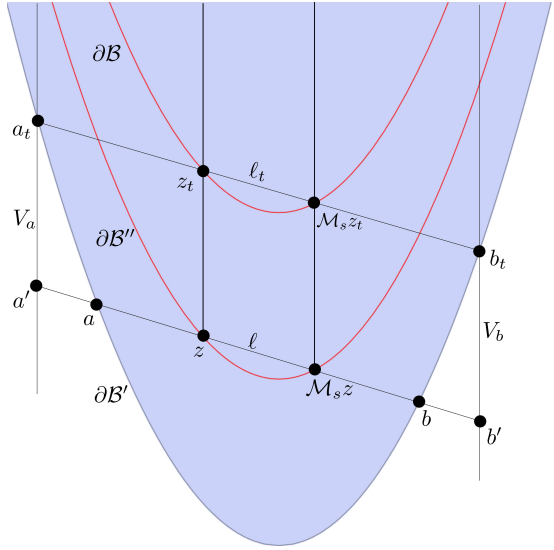}
  \caption{Distance estimates for $L'$-horoballs \label{f:horoballs}}
 \end{center}
\end{figure}

As a result we see that every point on $\partial \mathcal{B}''$ is moved the same distance in the Hilbert metric on $\mathcal{B}'$ by $\mathcal{M}_s$. For $t>0$ let $z_t=z+te_1$. Let $\ell$ be the affine line connecting $z$ and $\mathcal{M}_s z$ and let $a$ and $b$ be the intersection points of $\ell$ with $\partial \mathcal{B}'$. Similarly, let $\ell_t$ be the affine line connecting $z_t$ and $\mathcal{M}_s z_t$ and let $a_t$ and $b_t$ be the intersection points of $\ell_t$ with $\partial \mathcal{B}'$. Let $V_a$ and $V_b$ be the vertical lines passing through $a_t$ and $b_t$, respectively. Define $a'$ and $b'$ to be the respective intersection points of $V_a$ and $V_b$ with $\ell$. Figure \ref{f:horoballs} depicts this configuration and we see that 
$$[a_t:z_t:\mathcal{M}_s z_t: b_t]=[a':z:\mathcal{M}_s z:b']<[a:z:\mathcal{M}_s z:b],$$
and thus $\dhilb[\mathcal{B}']{z_t}{\mathcal{M}_s z_t}$ is a strictly decreasing function of $t$. Furthermore, as $t\to\infty$ we see that $[a_t:z_t:\mathcal{M}_sz_t:b_t]\to 1$, and so for sufficiently large $t$ we have $\dhilb[\mathcal{B}']{z_t}{\mathcal{M}_s z_t}<\mu$. Let $z_T$ be such that $\dhilb[\mathcal{B}']{z_T}{\mathcal{M}_s z_T}<\mu$ and let $\mathcal{B}$ be the $L'$-horoball such that $z_T\in \partial\mathcal{B}$. By construction we see that every point $z'\in \mathcal{B}$ is moved at most $\mu$ in the Hilbert metric on $\mathcal{B}'$ by $\mathcal{M}_s$. Since $\mathcal{B}\subset \Omega_s$ we have that $\dhilb[\Omega_s]{z'}{\mathcal{M}_s z'}\leq \dhilb[\mathcal{B}']{z'}{\mathcal{M}_sz'}<\mu$.  Let $\tau\in \Gamma_s$ and suppose that $u\in \tau\mathcal{B}\cap \mathcal{B}$. The proof will be complete if we can show that $\tau \in \Delta_s$.  Since $u\in \tau \mathcal{B}$ we can find $v\in \mathcal{B}$ such that $\tau v=u$. As a result we have 
$$\dhilb[\Omega_s]{u}{\tau \mathcal{M}_s \tau^{-1}u}=\dhilb[\Omega_s]{\tau v}{\tau \mathcal{M}_s v}=\dhilb[\Omega_s]{v}{\mathcal{M}_s v}<\mu.$$
The elements $\mathcal{M}_s$ and $\tau \mathcal{M}_s \tau^{-1}$ both move $u$ a distance less than $\mu$ in the Hilbert metric on $\Omega_s$. By the properly convex Margulis lemma we see that $\langle \mathcal{M}_s,\tau \mathcal{M}_s \tau^{-1} \rangle$ is virtually nilpotent.

Since $\Gamma_s$ is the fundamental group of a finite volume hyperbolic 3-manifold, it admits an action on $\HH^3$ and its ideal boundary. Consequently we see that element of $\Gamma_s$ commute if and only if they have the same fixed point set for this action. Since any virtually nilpotent subgroup of the fundamental group of a finite volume hyperbolic 3-manifold is abelian we see that $\mathcal{M}_s$ and $\tau\mathcal{M}_s\tau^{-1}$ commute. Since $\mathcal{M}_s$ is contained in a peripheral subgroup it acts as a parabolic isometry on $\HH^3$. This implies that $\mathcal{M}_s$ and $\tau\mathcal{M}_s\tau^{-1}$ have the same fixed point. Furthermore, $\tau$ also fixes the unique fixed point of $\mathcal{M}_s$. Since the action of $\Gamma_s$ on $\HH^3$ is properly discontinuous this implies that $\tau$ also has a single fixed point and we conclude that $\tau$ and $\mathcal{M}_s$ commute. Thus $\tau\in \Delta_s$.

\qed 
\end{proof}

\begin{remark}\label{r:precinvhoroball}
From the proof of Lemma \ref{l:precinvhoroball} we see that if $\mathcal{B}$ is an $L'$-horoball that is precisely invariant under $\Delta_s$ and $\mathcal{B}'$ is a $L'$-horoball such that $\mathcal{B}'\subset \mathcal{B}$, then $\mathcal{B}'$ is also precisely invariant under $\Delta_s$.
\end{remark}

We can now complete the proof of Theorem \ref{t:maintheorem}. 

\begin{proof}[Proof of Theorem \ref{t:maintheorem}]
From Proposition \ref{p:convexstructure} we know that for sufficiently small $s$ that $\rho_s$ is the holonomy of a properly convex projective structure on $M$ and that this structure is strictly convex if and only if $s=0$. Thus we can find $\Omega_s$ such that $M\cong \Omega_s/\Gamma_s$. The proof will be complete if we can show that $\Omega_s/\Gamma_s$ has finite Busemann volume. 

From Lemma \ref{l:precinvhoroball} and Remark \ref{r:precinvhoroball} we know that we can find $L'$-horoballs $\mathcal{B}$ and $\mathcal{B}'$ that are precisely invariant under $\Delta_s$ and such that $\mathcal{B}'\subset \mathcal{B}\subset \Omega_s$. Since $\mathcal{B}'$ is precisely invariant under $\Delta_s$ we see that $\mathcal{B}'/\Delta_s$ is an embedded submanifold of $\Omega_s/\Gamma_s$. The complement $(\Omega_s/\Gamma_s)\bs\left(\mathcal{B}'/\Delta_s\right)$ is compact and so the proof will be complete if we can show that $\mathcal{B}'/\Delta_s$ is a finite Busemann volume submanifold of $\Omega_s/\Gamma_s$.

Let $K$ be a fundamental domain in $\mathcal{B}'$ for the action of $\Delta_s$. By Proposition \ref{t:finitevolume} we know that $\hilbvol[\mathcal{B}]{K}<\infty$. However $\mathcal{B}\subset \Omega_s$ and so $\hilbvol[\Omega_s]{K}\leq \hilbvol[\mathcal{B}]{K}<\infty$. We conclude that $\mathcal{B}'/\Delta_s$ is a finite volume submanifold of $\Omega_s/\Gamma_s$. 
\qed 
\end{proof}

\bibliographystyle{spbasic}
\bibliography{bibliography}

\end{document}